\documentclass[a4paper,12pt,reqno]{amsart}
\usepackage{amsmath,amsfonts,amsthm,amssymb,epsfig, graphicx}

\usepackage{amsmath,amsfonts,amsthm,amssymb,color}
\usepackage[latin1]{inputenc}

\newcommand{\R}{\mathbb R}
\newcommand{\N}{\mathbb N}

\newcommand{\eps}{\varepsilon}

 \newcommand{\RR}{\mathbb{R}}
\newcommand{\EX}{\mathbf{E}}
\newcommand{\X}{ \overline{X}}
\newcommand{\Y}{ \overline{Y}}
\newcommand{\Z}{ \overline{Z}}

\newtheorem{theorem}{Theorem}[section]

\newtheorem{corollary}[theorem]{Corollary}

\newtheorem{example}[theorem]{Example}

\newtheorem{lemma}[theorem]{Lemma}

\newtheorem{proposition}[theorem]{Proposition}

\theoremstyle{remark}

\begin{document}
\title[Convergence of Numerical Methods for SDEs in Finance]{Convergence of Numerical Methods for Stochastic Differential Equations in Mathematical Finance}
\author{Peter Kloeden and Andreas Neuenkirch}

\address{Peter Kloeden, Institut f\"ur Mathematik, Johann Wolfgang Goethe-Universit{\"a}t,
Robert-Mayer-Strasse 10, D-60325 Frankfurt am Main, Germany     \\
{\tt kloeden@math.uni-frankfurt.de}}


\address{Andreas Neuenkirch, Institut f\"ur Mathematik, Universit\"at Mannheim, A5,6, D-68131 Mannheim, Germany\\
{\tt neuenkirch@kiwi.math.uni-mannheim.de}}


\date{\today}

\begin{abstract}
Many  stochastic differential equations that occur in financial modelling do not satisfy the standard assumptions made in convergence proofs 
of numerical schemes that are given in textbooks, i.e., their coefficients and the corresponding derivatives appearing in the proofs 
are not uniformly bounded and hence, in particular, not globally Lipschitz.  Specific examples are the Heston and Cox-Ingersoll-Ross models with square root coefficients 
and the Ait-Sahalia model with rational coefficient functions. Simple examples show that, for example,  the Euler-Maruyama scheme may not converge either 
in the strong or weak sense when the standard assumptions do not hold. Nevertheless, new convergence results have been obtained recently for many such  models 
in financial mathematics. These are reviewed here. Although weak convergence is of traditional importance in financial mathematics with its emphasis on 
expectations of functionals of the solutions, strong convergence  plays a crucial role in Multi Level Monte Carlo methods, so it and also pathwise convergence 
will  be considered along with methods which preserve the positivity of the solutions.
 \end{abstract}

\maketitle

\section{Introduction}

Consider the It\^{o} stochastic differential equation (SDE) in $\mathbb{R}^{d}$
\begin{equation}\label{itosde}
dX_t = a(X_t)dt + \sum_{j=1}^{m} b_{j}(X_t)  dW^{(j)}_t,  \quad t \in [0,T], \qquad X_0=x_0 \in \mathbb{R}^d
\end{equation}
with drift and diffusion coefficients $a$, $b_{j}$ $:$ $\mathbb{R}^{d}$ $\rightarrow$ $\mathbb{R}^{d}$
for $j=1, \ldots, m$. Here  $W_t$ $=$ $(W_t^{(1)}, 
\ldots, W_t^{(m)})$, $t \geq 0$, is  an $m$-dimensional   Brownian motion on a probability space $(\Omega, \mathcal{F}, \mathbf{P})$ and  superscripts in brackets
 label components of vectors.  Throughout this article it will always be assumed that equation (\ref{itosde}) has a  unique strong solution.

Explicit solutions of such equations are rarely known, thus one has to rely on numerical methods to simulate their sample paths $X_t(\omega)$ or to estimate functionals $\mathbf{E}\Phi(X)$ 
for some  $\Phi: C([0,T];\mathbb{R}^d) \rightarrow \mathbb{R}$. Typically, such a numerical method   relies on a discretization
 $$0 \leq t_1 \leq t_2 \leq  \ldots \leq t_n =T$$  
 and a global approximation on $[0,T]$ is obtained by interpolation.
 
In the case
of the  classical weak approximation  the
error of an approximation $\overline{X}$ to $X$  is  measured by the quantity
$$     |\mathbf{E} \phi(X_T) - \mathbf{E} \phi(\overline{X}_T) | $$
for smooth functions  $\phi: \mathbb{R}^{d} \rightarrow \mathbb{R}$. The test functions $\phi$ are a particular case of the general (path-dependent) functionals $\Phi$.
In the strong approximation problem  the $p$-th mean of the difference between $X$
and $\overline{X}$  is analyzed, i.e.
$$    \Big(\mathbf{E} \sup_{k=0, \ldots, n} | X_{t_{k}} - \overline{X}_{t_{k}} |^{p} \Big)^{1/p}$$
for the maximal error in the discretization points or
$$    \Big(\mathbf{E} \sup_{t \in [0,T]} | X_{t} - \overline{X}_{t} |^{p} \Big)^{1/p}$$
for the global error,
where $p \geq 1$ and $| \cdot |$ denotes the Euclidean norm. Here
 the mean-square error, i.e. $p=2$,  is usually studied.  The  recent development of the Multi-level Monte Carlo method for SDEs \cite{g1,g2} has revealed that strong error bounds are crucial for the efficient computation of functionals  $\mathbf{E}\Phi(X)$.

While the strong error measures the error of the approximate sample paths $\overline{X}$ on average,
the pathwise error is the random quantity   
$$   \sup_{k=0, \ldots, n}|X_{t_{k}}(\omega) - \overline{X}_{t_{k}}(\omega)|, \qquad \omega \in \Omega$$ and
$$   \sup_{t \in [0,T] }|X_{t}(\omega) - \overline{X}_{t}(\omega)|, \qquad \omega \in \Omega$$ respectively.
  Here the error is analyzed
for a fixed $\omega \in \Omega$ without averaging. This quantity thus gives the error of the actually calculated   approximation  $\overline{X}_{t_k}(\omega)$, $k=0, \ldots,n$, respectively $\overline{X}(\omega)$.

The traditional weak and strong convergence analysis for numerical methods for stochastic differential equations (SDEs) relies on the global Lipschitz assumption, i.e. the SDE coefficients satisfy
\begin{align*}
|a(x)-a(y)| + \sum_{j=1}^m |b_{j}(x)-b_j(y)| \leq L \cdot |x-y|, \qquad x,y \in \mathbb{R}^d            
\end{align*}
 for some $L>0$.
 However, in many SDEs used for modelling in mathematical finance
 this assumption is violated, so the standard results (see \cite{KP,M})
 do not  apply. \\

The Constant Elasticity of Variance Model for asset prices \cite{cox:1975} , which was introduced by Cox in 1975, is given by the SDE
\begin{align*} 
dS_t=\mu S_t \, dt + \sigma S_t^{\gamma} \, dW_t, \qquad S_0=s_0>0
\end{align*}
where $\mu \in \R$, $\sigma  >0$ and $\gamma \in (0,1]$ and $W_t, t \geq 0,$ is a one-dimensional Brownian motion. For $\gamma =1$ this is the standard Black-Scholes model (i.e. a geometric Brownian motion),
while for  $\gamma \in (0,1)$ the diffusion coefficient of this SDE is clearly not globally Lipschitz continuous. 
 This SDE has a unique strong solution if and only if $\gamma \in   [1/2,1]$ and takes values in $[0,\infty)$.

The Ait-Sahalia model and its generalization  \cite{AS,SMHJ}, which are  stochastic interest rate models,
follow the dynamics
\begin{align*}
dX_t =  \big(\alpha_{-1} X_t^{-1} - \alpha_0 + \alpha_1 X_t - \alpha_2X_t^r \big)dt +\sigma X_t^{\rho} dW_t, \qquad X_0=x_0>0
\end{align*}
where  $ \alpha_{i}, \sigma, r, \rho > 0$, $i=-1, \ldots, 2$. Under certain conditions on the parameters (see \cite{SMHJ}), this SDE has a unique strong solution with values in $(0, \infty)$.  Note that here the diffusion coefficient  grows superlinearly for large values of $x$ while the drift coefficient  has a singularity at $x=0$.

The Heston model \cite{heston}, which is an asset price model with stochastic volatility, is another example for an SDE with non-Lipschitz coefficients. This SDE takes non-negative values  only and contains square root coefficients:
\begin{align*}
dS_t &= \mu
S_t \, dt + \sqrt{V_t} S_t  \, \big( \sqrt{1-\rho^2} \, d W^{(1)}_t + \rho\, dW^{(2)}_t\big),&  \quad S_0=s_0  > 0   
\\ 
dV_t &= \kappa( \lambda -  V_t) \, dt +  \theta \sqrt{V_t}  \, dW^{(2)}_t, & \quad V_0=v_0 > 0. 
\end{align*} The parameters satisfy $\mu \in \mathbb{R}$, $\kappa, \lambda, \theta > 0$ and  $\rho \in (-1,1)$.  The second component of this SDE is the Cox-Ingersoll-Ross process, which is also used as a short rate model  \cite{cox:1985}.

Finally, the use of the inverse of the CIR process as volatility process leads to the so-called $3/2$-model 
\begin{align*}
dS_t &= \mu
S_t \, dt + \sqrt{V_t} S_t   \, \big( \sqrt{1-\rho^2} \, d W^{(1)}_t + \rho\, dW^{(2)}_t\big),&  \quad S_0=s_0  > 0  
\\
dV_t &=  c_1 V_t( c_2 -  V_t) \, dt +  c_3 V_t^{3/2}  \, dW^{(2)}_t, & \quad V_0=v_0>  0
\end{align*} where $c_1,c_2,c_3 >0$,
see e.g. \cite{heston_2}.\\

Motivated by these and other examples, the investigation of  numerical methods for SDEs with non-Lipschitz coefficients has been an active field of research in recent years.  This  article, 
provides an overview of the new  developments using the above equations as illustrative examples and discussing, in particular, Euler-type schemes. 
For some of the above equations exact simulation methods exist, see e.g. \cite{bk,gm} and also \cite{roberts} for a  class
of one-dimensional equations,
which are superior for the simulation of the SDEs at a single or a few time points. 
However, if  a full sample
path of the SDE has to be simulated or if the  SDEs under consideration are
part of a larger SDE system, then discretization schemes are typically more efficient.

\bigskip
\bigskip

\section{Pathwise Convergence Rates of the Euler Scheme and  general It\^o-Taylor Methods}

The pathwise error criteria are very robust with respect to the global Lipschitz assumption.
One of the simplest approximation schemes for  equation (\ref{itosde}) is the Euler scheme
\begin{align*}
\overline{X}_{t_{k+1}} &= \overline{X}_{t_{k}}  + a(\overline{X}_{t_{k}}) \Delta + \sum_{j=1}^m b_j (\overline{X}_{t_{k}}) \Delta_k W^{(j)}, \qquad k=0,1, \ldots,    
\end{align*}
with $\overline{X}_{0}  =x_0$,  where  $\Delta =T/n$, $t_k=k\Delta$ and $\Delta_k W = W_{t_{k+1}}-W_{t_k}$.
The Euler scheme (and all other approximation methods that will be introduced below)
depend on the stepsize $\Delta>0$, hence  on $n \in \mathbb{N}$, but  this dependence will be omitted whenever it is clear from the context.

From the results of Gy\"ongy  \cite{gyoengy}  it follows that the Euler scheme has pathwise convergence order $1/2-\varepsilon$ also if the SDE coefficients are only locally Lipschitz continuous: 
for all $\varepsilon >0 $ 
$$  \sup_{k=0, \ldots, n}|X_{t_{k}} - \overline{X}_{t_{k}}| \leq \eta_{\eps}^{E}  \cdot n^{-1/2 + \varepsilon} $$ almost surely   
for  a finite and non-negative random variable $\eta_{\varepsilon}^E$
under the assumption that for all $N \in \mathbb{N}$ there exist constants $L_N>0$ such that
\begin{align*} |a(x)-a(y)| + \sum_{j=1}^m |b_{j}(x)-b_j(y)| \leq L_N \cdot |x-y|, \qquad |x|,|y| \leq N.             
\end{align*}
Thus, the pathwise convergence rate of the Euler scheme coincides up to an arbitrarily small $\varepsilon >0$ with its strong convergence rate $1/2$, but for the pathwise convergence rate
 no global Lipschitz assumption is required.

Jentzen, Kloeden \& Neuenkirch  \cite{JKN}  observed that this is not a specific feature of the Euler scheme but, in fact,  holds for general It\^o-Taylor schemes of order $\gamma = 
0.5, 1.0, 1.5, \ldots$. For the  definition of these 
schemes, see e.g. \cite{KP}. The Euler scheme corresponds to $\gamma=0.5$, while $\gamma=1.0$ yields the Milstein scheme
\begin{align*}
 \overline{X}_{t_{k+1}} &= \overline{X}_{t_{k}}  + a(\overline{X}_{t_{k}}) \Delta + \sum_{j=1}^m b_j (\overline{X}_{t_{k}}) \Delta_k W^{(j)}  + \sum_{j_1,j_2=1}^m L^{j_1}b_{j_2}(\overline{X}_{t_{k}}) I_{j_1,j_2}(t_k,t_{k+1})
\end{align*}
with the differential operators
$$ L^{j}  =\sum_{k=1}^{d} b_j^{(k)} \frac{\partial}{\partial x^{k}}, \qquad j = 1 , \ldots, m  $$
and the iterated It\^o-integrals
$$ I_{j_1,j_2}(s,t) = \int_{s}^{t} 
\int_{s}^{\tau_2} d W^{(j_{1})}_{\tau_{1}} \,  d W^{(j_2)}_{\tau_2}, \qquad j_1,j_2=1, \ldots, m.$$
The It\^o-Taylor scheme of order $1.5$ is usually  called the Wagner-Platen scheme.

\smallskip

\begin{theorem}\label{thm_pw_1}  Let  $\gamma$ $=$ $0.5$, $1.0$, $1.5$, $\ldots$. Assume that $a$, $b_1$, $\ldots$, ${b_m}$  $\in$ $C^{ 2 \gamma  +1}(\mathbb{R}^d ;\mathbb{R}^{d})$ and moreover let $\X^{\gamma ,n}$ be the It\^o-Taylor scheme of order $\gamma$ with stepsize $\Delta =T/n$.
Then  for every $\varepsilon >0$  there exists a  
non--negative random variable $\eta_{\varepsilon}^{\gamma} $  such that
$$ 
\sup_{k=0, \ldots, n}  \left |X_{t_k}(\omega ) -\X^{\gamma,n}_{t_k}(\omega)
\right| \leq \eta_{ \varepsilon}^{\gamma}(\omega)
\cdot n^{-\gamma + \varepsilon} 
$$ 
for almost all $\omega \in \Omega$.
\end{theorem} 

\smallskip

The main ingredients to obtain this result  are the Burkholder-Davis-Gundy inequality, which implies that 
all moments of an It\^o-integral are equivalent, the following  Borel-Cantelli-type Lemma, and a localization  procedure.

\smallskip

\begin{lemma}(see \cite{lms})
Let $\alpha >0$, $c_p \geq 0$ for $p \geq 1$ and let
$(Z_{n})_{n \in \mathbb{N}}$  be a sequence of  random variables with 

$$( {\bf E}
|Z_{n}|^{p})^{1/p} \leq c_p \cdot n^{-\alpha}
$$ 
for all $p \geq 1$ and $n \in \mathbb{N}$. Then for every
$\varepsilon > 0$ there  exists a  finite and non-negative random variable
  $\eta_{\eps}$ such that
$$ |Z_{n}| \leq \eta_{\varepsilon}\cdot n^{-\alpha + \varepsilon} $$
almost surely for all $n \in \mathbb{N}$.
 \end{lemma}

\smallskip

The Burkholder-Davis-Gundy inequality and the Borel-Cantelli-type Lemma allow  one  to show that
the It\^o-Taylor scheme of order $\gamma$ has pathwise convergence rate $\gamma- \varepsilon$ 
for smooth and bounded coefficients with bounded derivatives, thereby  extending the classical mean-square convergence analysis in \cite{KP}. Then a localization argument
is applied to avoid the boundedness assumptions. Roughly speaking, this localization argument works as follows: A fixed sample path $X_t(\omega), \, t \in [0,T],$ of the SDE solution is bounded, 
i.e. stays in some open set $B(\omega)$. However for the SDE 
$$ dY_t =\widetilde{a}(Y_t)\,dt + \sum_{j=1}^m \widetilde{b}_j(Y_t)\, d W^{(j)}_t, \qquad Y_0=x_0$$
with smooth and bounded coefficients $\widetilde{a}$, $\widetilde{b}_j$ with bounded derivatives, which coincide with the ones of the original SDE on $B(\omega)$, the solution sample path 
$Y_t(\omega), t \in [0,T]$, coincides with $X_t(\omega), t \in [0,T]$. Asymptotically this also holds
for the corresponding sample paths of the $\gamma$-It\^o-Taylor schemes, so the pathwise convergence rates carry over.

\bigskip

Note that all the  examples of SDEs given in the introduction take non-negative values only, so good approximation schemes should preserve  this structural property. The (explicit) Euler scheme is, in 
general, not such a scheme, since its increments are conditionally Gaussian.
For example, in case of the CIR process
$$dX_t= \kappa(\lambda- X_t) \,d t +  \theta \sqrt{X_t} \, dW_t, \qquad X_0=x_0 >0$$ the transition density of the Euler scheme reads as
$$ p(y;x)= \frac{1}{  \sqrt{2 \pi \theta^2 x \Delta}} \exp \left( -\frac{\big(y -( x +\kappa(\lambda -x) \Delta \big)^2}{2 \theta^2 x \Delta}\right), \qquad y \in \mathbb{R}, \, x >0,$$ 
so negative values can be obtained with positive probability even in the first step. This  has lead to many ad-hoc corrections to prevent termination of the Euler scheme. The   truncated 
Euler scheme 
\begin{align} \label{euler_2}
\X_{t_{k+1}} &= \X_{t_k} +  \kappa( \lambda - \X_{t_k} ) \, \Delta + \theta  \sqrt{ \X_{t_k}^+ } \, \Delta_k W  , \qquad k=0, 1, \ldots
\end{align}   was proposed in \cite{delbaen}, while  the  scheme 
\begin{align} \label{euler_3}
\X_{t_{k+1}} &= \X_{t_k} +  \kappa( \lambda- \X_{t_k} )  \Delta + \theta  \sqrt{ |\X_{t_k}| } \, \Delta_k W  , \qquad k=0, 1, \ldots
\end{align}
was  studied in \cite{hm}. Both approaches extend the mapping $[0, \infty) \ni x \mapsto \sqrt{x} \in [0, \infty)$  suitably to negative values of $x$. For  the CIR process
 this idea was taken further by Lord, Koekkoek \& 
van Dijk \cite{lkd}, who also proposed modifications of the drift coefficient for negative values of the state space.\\

\begin{example}\label{cir_scen_def} {\rm 
The following  table shows the average number of negative steps  per path  for the above Euler approximations of the CIR process.  Scenario I (taken from \cite{AS_K}), corresponds to the parameters
$$ x_0 = 0.05 , \qquad \kappa =5.07 , \qquad \lambda = 0.0457, \qquad \theta = 0.48, \qquad T=5$$
while  Scenario II (taken from \cite{bk}) uses $$ x_0=0.09, \qquad \kappa = 2, \qquad \lambda =0.09, \qquad \theta =1, \qquad T=5.$$ The stepsize for the Euler schemes is given by $\Delta =T/n$ with $n=512$.
 \medskip
 \begin{center}
\begin{tabular}{c||c|c} 
average negative steps of / for & Scenario I & Scenario II \\ \hline \hline
Euler scheme  \eqref{euler_2} & 0.9141 &  64.8611 \\ \hline
Euler scheme \eqref{euler_3}  &  1.0590 & 74.5017
\end{tabular}
\end{center}
\medskip
The empirical frequency of negative paths is  0.4913 in Scenario I and 0.9990 in Scenario II.  These results were obtained by a Monte Carlo simulation with $N=10^6$ repetition. They clearly  indicate that the Euler scheme \eqref{euler_3} has a tendency for negative ``excursions". This can also  be seen  in Figure \ref{figure_euler_exc}, which shows a sample path  of the (linearly interpolated) Euler  schemes \eqref{euler_2} and \eqref{euler_3} using the same path of the driving Brownian motion. The parameters used in this figure correspond to Scenario II.  \begin{flushright} $\diamond$ \end{flushright}
\begin{center}
\begin{figure}[h]
\ \leavevmode \epsfig{figure=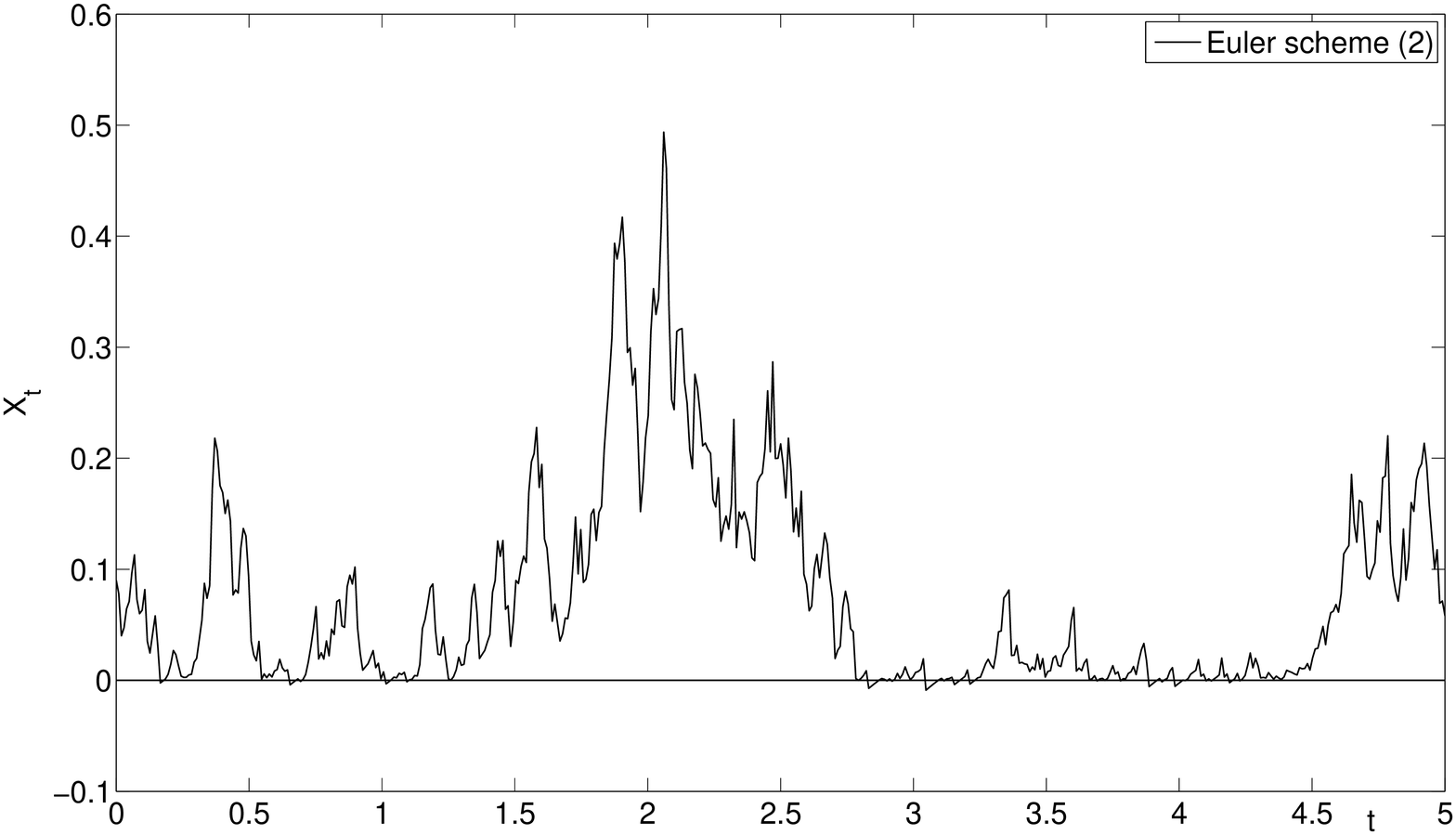, width=13cm,
height=8cm} \hspace{0.1cm}
\ \leavevmode \epsfig{figure=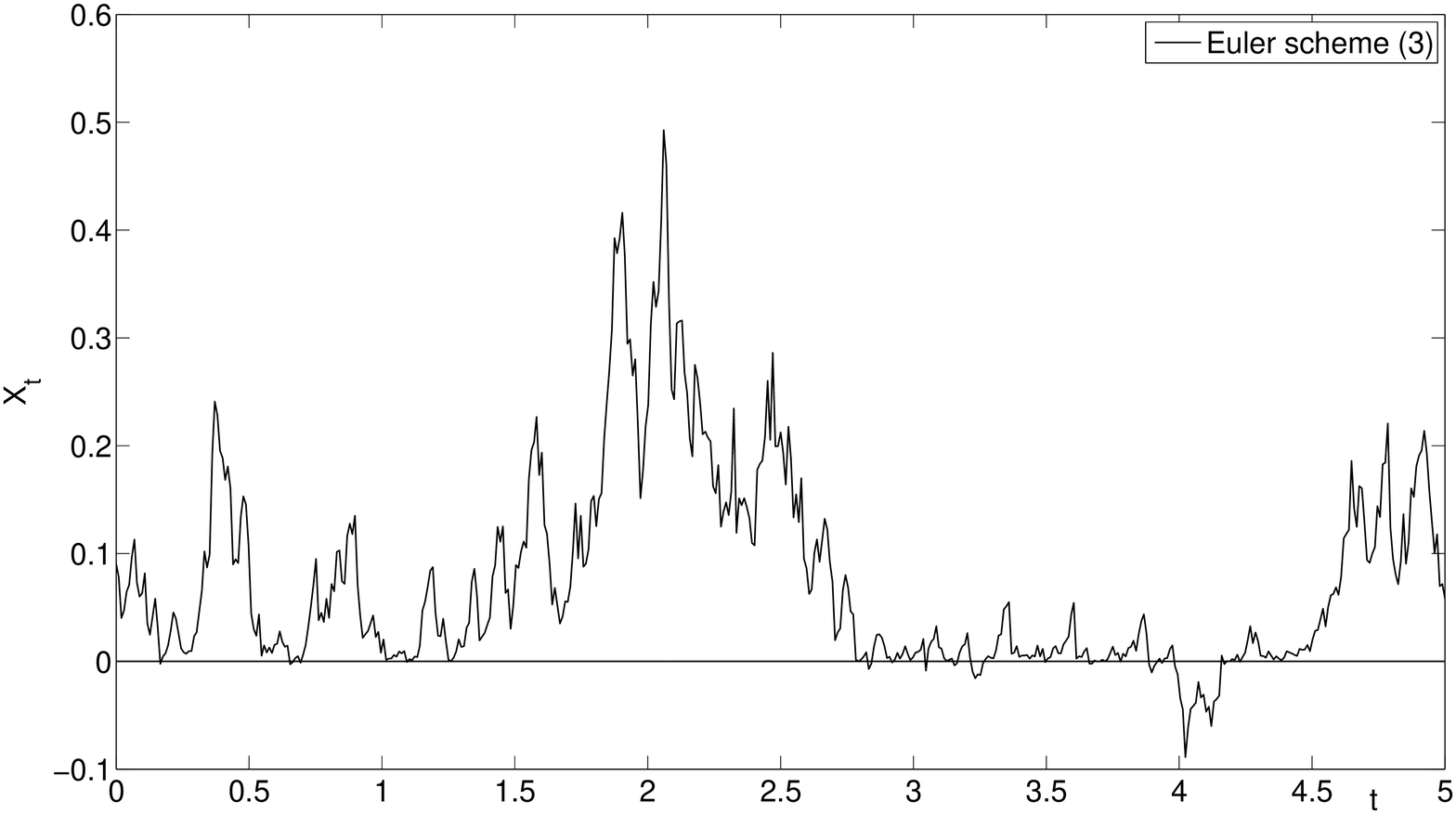, width=13cm,
height=8cm}  \caption{\label{figure_euler_exc}A path of Euler scheme \eqref{euler_2} vs.  Euler scheme \eqref{euler_3} for the CIR process and Scenario II}
\end{figure}  
\end{center}
}
\end{example}  
  
\bigskip
\smallskip

For general SDEs the procedure of modifying the coefficients outside the support of the solution has been introduced systematically in \cite{JKN}. For an SDE
\begin{equation}  \label{snew.sde2}
dX_t = a(X_t)\, dt + \sum_{j=1}^{m} b_{j}(X_t)\,   dW^{(j)}_t,   \qquad X_0=x_0
\end{equation}
which takes values in a  domain $D$ $\subset$ $\mathbb{R}^{d}$, i.e.
\begin{align} \mathbf{P}(X_t \in D, \,\, t \geq 0)=1,  \label{support_sde} \end{align}
the auxiliary coefficients 
\begin{align*}
\widetilde{a}(x) &=  a(x) \cdot  \mathbf{1}_{D}(x) + f(x) \cdot\mathbf{1}_{E}(x), & \qquad x \in \mathbb{R}^{d}  &
\\[2ex]
\widetilde{b}_j(x)&=  b_j(x) \cdot  \mathbf{1}_{D}(x) + g_j(x) \cdot \mathbf{1}_{E}(x), &\qquad x \in \mathbb{R}^{d}, &\quad j=1, \ldots, m 
\end{align*}
with $E= \mathbb{R}^d \setminus D$ are introduced there.
A modified It\^{o}-Taylor scheme of order $\gamma$ based on the auxiliary functions $f$ and $g$ is then the corresponding standard It\^{o}-Taylor scheme
for the SDE
\begin{equation*} 
dX_t = \widetilde{a}(X_t)\, dt + \sum_{j=1}^{m} \widetilde{b}_{j}(X_t)\,  dW^{(j)}_t,  
\end{equation*}
with  a suitable definition of the  derivatives of the coefficients on $\partial D$, see  \cite{JKN} for details. This method is well-defined as long as the coefficients of the
 equation  are  $(2\gamma+1)$-times differentiable on $D$ and the auxiliary functions are $(2\gamma-1)$-times differentiable on $E$.
The purpose of the auxiliary functions is twofold: to obtain a well-defined approximation scheme
and to bring the numerical scheme back to $D$ if it leaves  $D$. In particular, the auxiliary functions can always be chosen to be affine or even constant. It was shown by Jentzen,
 Kloeden \& Neuenkirch \cite{JKN}  that  Theorem \ref{thm_pw_1} adapts to  modified It\^{o}-Taylor schemes for SDEs on  domains    $ D \subset \mathbb{R}^{d}$.

\smallskip

\begin{theorem} \label{taylor_mod} 
Let $X$  be the solution of SDE \eqref{snew.sde2} satisfying condition \eqref{support_sde}.
Moreover let $\gamma = 0.5, 1.0, 1.5, \ldots$ and  assume that
 $$a\in C^{ 2   \gamma  +1 }( D ;
  \R^{d}), \qquad b \in C^{ 2   \gamma +1 }( D ;
  \R^{d,m}) $$
and $$  f \in    C^{2
  \gamma  -1 }( E ;
  \R^{d} ), \qquad g \in   C^{       2
 \gamma  -1}( E ;
  \R^{d,m}) .
     $$
Finally let  $ \widetilde{X}^{\gamma,n}$ be  the modified
  It\^{o}-Taylor method for $X$  based on the auxiliary functions $f$ and $g$ with stepsize $\Delta=T/n$. 
Then for every $\varepsilon >0$ there exists a finite and non-negative
  random variable $\eta_{\gamma, \varepsilon}^{f,g} $ such that
$$ 
\sup_{k=0, \ldots, n}  \big |X_{t_k}(\omega ) -\widetilde{X}^{\gamma,n}_{t_k}(\omega)
\big | \leq \eta_{\gamma, \varepsilon}^{f,g}(\omega)
\cdot n^{-\gamma + \varepsilon}  $$
for almost all $\omega \in \Omega$ and all $n \in \N$.
\end{theorem}

\smallskip

In the case of the Euler scheme, i.e. $\gamma =0.5$, the assumptions on $a$ and $b$ can be weakened to the assumption that $a$ and $b$ are locally Lipschitz continuous on $D$.
 For SDEs on domains in mathematical finance this condition is typically satisfied. In fact, in most cases the coefficients are infinitely differentiable.

\smallskip

The CIR process satisfies
$$ \mathbf{P}(X_t > 0 \,\,\, {\textrm{for all}} \, \,\, t \geq 0)=1$$ if and only if $2 \kappa \lambda \geq \theta^2$.
The latter  assumption is typically satisfied in interest rate applications of the CIR process.
Hence,  modified Taylor schemes can be used here with $D=(0, \infty)$.  
The truncated Euler scheme (\ref{euler_2}) corresponds to the auxiliary functions $f(x)=a(x)$, $g(x)=0$, $x \leq 0$, while the scheme (\ref{euler_3}) uses  the auxiliary 
functions $f(x)=a(x)$, $g(x)=\sqrt{-x}$, $x \leq 0$.  Note that $2\kappa \lambda \geq \theta^2$ is satisfied in Scenario I, but not in Scenario II.

\bigskip

For structure preserving integration of the CIR process  also the symmetrized Euler method 
\begin{align}\label{euler_1}
\widetilde{X}_{t_{k+1}} &= \Big | \widetilde{X}_{t_k} +  \kappa( \lambda
  - \widetilde{X}{t_k} )  \Delta  + \theta  \sqrt{
     \widetilde{X}_{t_k}} \, \Delta_k W \Big|, \qquad k=0,1, \ldots  
\end{align} 
was   proposed in \cite{diop,diop_2}. While the modified Euler schemes \eqref{euler_2} and \eqref{euler_3} may leave $(0, \infty) $ and are then forced back in the next 
steps, this scheme is always non-negative. Adapting this to general SDEs, which take values in a domain $D$, leads to 
the  reflected Euler schemes, see e.g. \cite{NZ}, which are given by 
\begin{align*}
 \widetilde{X}^{\psi}_{t_{k+1}}  =   H^{\psi}_{t_{k+1}}  \cdot \mathbf{1}_{D}( H^{\psi}_{t_{k+1}}) + \psi(H^{\psi}_{t_{k+1}})  \cdot \mathbf{1}_{\mathbb{R}^d \setminus D}( H^{\psi}_{t_{k+1}}) 
\end{align*}
with $ \widetilde{X}^{\psi }_{0} = x_{0},$  where
\begin{align*} 
H^{\psi}_{t_{k+1}} = \widetilde{X}^{\psi}_{t_{k}} +  a(\widetilde{X}^{\psi}_{t_{k}}) \Delta + \sum_{j=1}^m b_j (\widetilde{X}^{\psi}_{t_{k}}) \Delta_k W^{(j)}
\end{align*}
and a  measurable projection function $\psi: \mathbb{R}^d \setminus D \rightarrow D \cup \partial D$.
A straightforward modification of the above theorem yields a pathwise convergence order $1/2-\varepsilon$  for these reflected Euler schemes if the SDE coefficients are twice 
continuously differentiable on $D$.
In the same way reflected It\^o-Taylor schemes of arbitrary order can be constructed and analyzed.

\smallskip

The symmetrized Euler scheme \eqref{euler_1} corresponds to the reflection function $\psi(x)=|x|$.  The results on modified It\^o-Taylor schemes and reflected 
Euler methods apply also to the generalized Ait-Sahalia model with $D=(0,\infty)$ if $r>1, \rho < (1+r)/2$, to the Heston model with $D=(0,\infty)^2$ if $2 \kappa \lambda \geq \theta^2$ and to the 3/2-model with 
$D=(0,\infty)^2$
 and no further 
restrictions on the parameter.

\smallskip

\begin{example}
{\rm   To illustrate the above results  consider Scenario I for the  Cox-Ingersoll-Ross process. 
Figure \ref{figure_cir_path} shows  for two  different sample paths
$\omega  \in \Omega$  the maximum error in the discretization points,  
i.e.  $$ \sup_{k=0, \ldots, n} |X_{t_k}(\omega)-\overline{X}_{t_{k}}
(\omega)|,$$  of 
\begin{itemize}
\item[(i)] the truncated Euler scheme \eqref{euler_2}
\item[(ii)] the symmetrized Euler scheme \eqref{euler_1}
\item[(iii)] the modified Milstein scheme with auxiliary functions
$f(x)=\kappa(\lambda- x)$, $g(x)=0$, i.e. a truncated Milstein scheme.
\end{itemize}
To estimate  the pathwise 
maximum error  for the above
 approximation schemes  the Cox-Ingersoll-Ross process  have been discretized with a very small step size using scheme \eqref{euler_2}.
\begin{figure}[htp]
 \leavevmode  \epsfig{figure=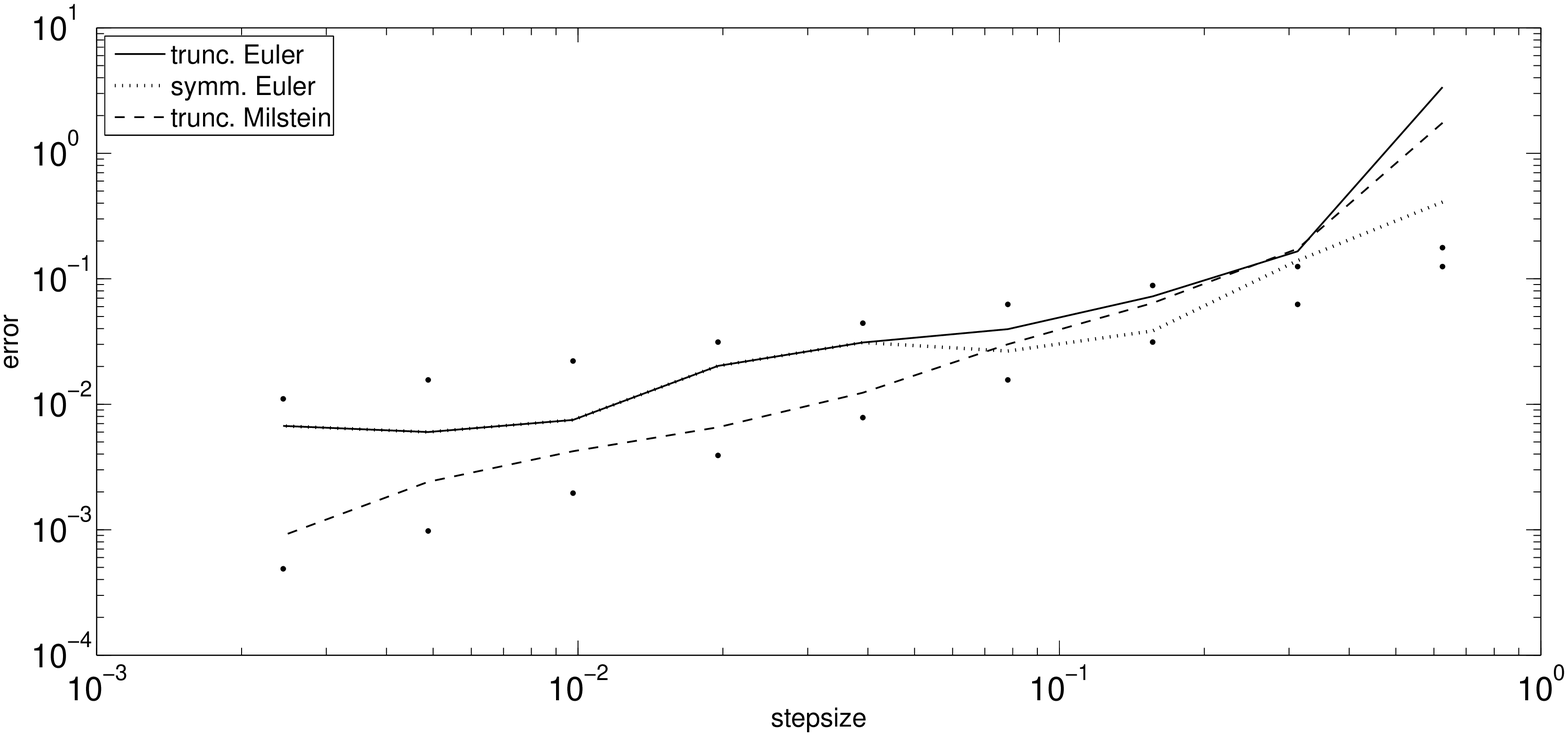, width=13cm,
height=8cm}  
\ \leavevmode   \epsfig{figure=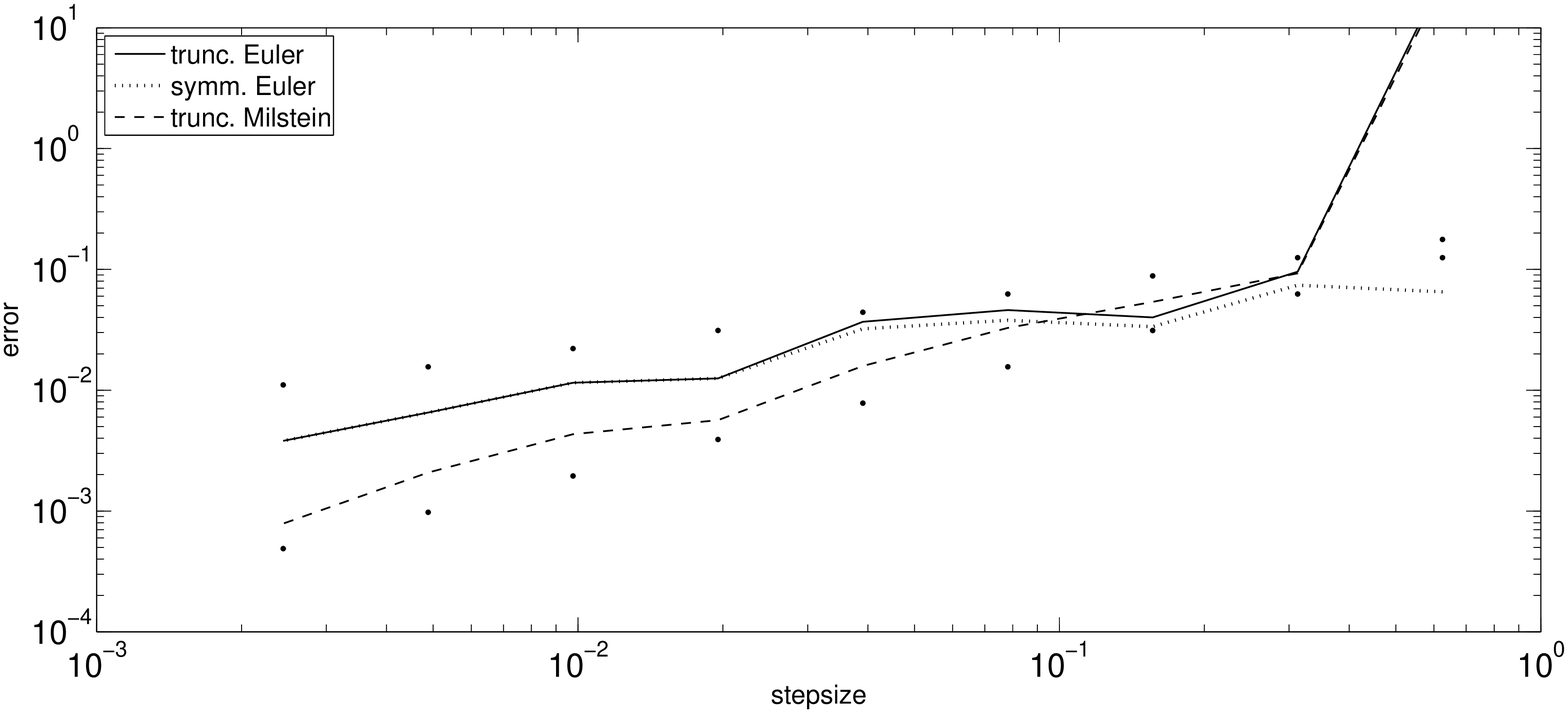, width=13cm,
height=8cm} 
\caption{\label{figure_cir_path}Pathwise maximum error  vs.~step size for two
   sample paths for the Cox-Ingersoll-Ross model for Scenario I}
\end{figure}
In Figure \ref{figure_cir_path}  
log-log-coordinates are used, so  the dots indicate the
convergence orders  $0.5$ and $1$. 
The pathwise  convergence rates of all three
approximation schemes are  in  good accordance
with the theoretically predicted rates for moderate and small step sizes. For small step sizes  both Euler schemes do not take negative values
and hence coincide. Moreover, for small step sizes
the Milstein scheme is superior due to its first order convergence.
 \begin{flushright} $\diamond$ \end{flushright}}
\end{example} 
\smallskip

Numerical methods with pathwise convergence rates of high  order are thus available also for SDEs with non-globally Lipschitz coefficients. However, while pathwise convergence rates are very important
 for the analysis of random dynamical systems \cite{arnold,GAKN}, one of the main objectives in mathematical finance  is the pricing of
(path-dependent) European-type derivatives, which means to compute  real numbers
$ \EX \Phi(X)$
where $\Phi : C([0,T]; \RR^d) \rightarrow \RR$ is the  discounted payoff of the derivative.   Since the  integrability of the random constants in the error bounds is an open problem, the above
 pathwise convergence rates  do not imply weak or strong convergence rates. Nevertheless,  if $\Phi$ is bounded and continuous and if $\X^{\gamma}=(\X^{\gamma}_t)_{ t \in [0,T]}$ is the piecewise linear  interpolation of the $\gamma$-It\^o-Taylor scheme
(standard, modified or reflected) then 
$$ \mathbf{E} \Phi(\X^{\gamma}) \longrightarrow  \mathbf{E} \Phi(X) $$
for $n \rightarrow \infty$, so for bounded and continuous pay-offs (e.g. put options) one obtains at least the convergence of the corresponding standard Monte Carlo estimators  for the option price. The same is true for  barrier options with payoff 
of the form
$$ \Phi(X)= \phi(X_T) \mathbf{1}_{ \{  K_1 \leq |X_t| \leq K_2, \,\,  t \in [0,T] \} } $$
with $0 \leq K_1 \leq K_2 < \infty$,
if $\phi$ is bounded and continuous  and the law of $\sup_{t \in [0,T]} |X_t|$ and  $\inf_{t \in [0,T]} |X_t|$ has a density with respect to the Lebesgue measure.

 \bigskip
 \bigskip

\section{The Explicit Euler Scheme: Criteria for Weak and Strong Convergence}
It was shown by Higham, Mao \& Stuart in \cite{HMS} that the explicit Euler scheme 
\begin{align*} \overline{X}_{t_{k+1}} &= \overline{X}_{t_{k}}  + a(\overline{X}_{t_{k}}) \Delta + \sum_{j=1}^m b_j (\overline{X}_{t_{k}}) \Delta_k W^{(j)}, \qquad k=0,1, \ldots,   
\\ \overline{X}_{0} &=x_0  \nonumber 
\end{align*}
is strongly convergent if the coefficients are locally Lipschitz continuous on $\RR^d$ and  a moment condition for the SDE and its  Euler approximation is satisfied. This result can be extended to SDE on domains and the modified or reflected Euler scheme.

\smallskip

\begin{theorem}\label{euler_hms_thm} 
Let $X$  be the solution of SDE \eqref{snew.sde2} satisfying condition \eqref{support_sde}.
Moreover, let  $ \widetilde{X}^{n}$ be  the modified Euler scheme based on the auxiliary functions $f \in C(E; \RR^d)$, $g \in C(E; \RR^{d,m}) $ with stepsize $\Delta=T/n$ or let  $ \widetilde{X}^{n}$ be the reflected Euler scheme based on the projection function $\psi : E \rightarrow D \cup \partial D$  with stepsize $\Delta=T/n$. 
Assume that
 $$a\in C^{ 2 }( D ;
  \R^{d}), \qquad b \in C^{ 2 }( D ;
  \R^{d,m}) $$
and
furthermore, assume that for some $p >2$  
\begin{align} \sup_{n \in \mathbb{N}} \, \mathbf{E} \max_{k=0, \ldots n} |\widetilde{X}^n_{t_k} |^p +  \mathbf{E}   \sup_{t \in [0,T]} |X_t|^p < \infty. \label{euler_moment}\end{align} 
Then  
$$ \lim_{n \rightarrow \infty}  \mathbf{E}  \max_{k=0, \ldots, n} |X_{t_k}- \widetilde{X}^n_{t_k} |^2=0. $$
\end{theorem}

\smallskip
\begin{proof}
From the results of the previous section  
$$ \lim_{n \rightarrow \infty}  \max_{k=0, \ldots, n} |X_{t_k}- \widetilde{X}^n_{t_k} |= 0 $$ 
hold, almost surely.  However,  
assumption (\ref{euler_moment}) implies the uniform integrability of  $$\max_{k=0, \ldots n} |X_{t_k}- \widetilde{X}_{t_k}^{n} |^2, \,\,\,\, n \in \mathbb{N}. $$ The  assertion now follows, since uniform
 integrability allows integration to the limit.
\end{proof}

\smallskip

Note that  assumption \eqref{euler_moment} is easily verified if the SDE coefficients have linear growth on $D$, i.e.
$$  |a(x)| + \sum_{j=1}^m |b_{j}(x)| \leq C \cdot (1+ |x|) , \qquad x \in D,   $$
for some $C>0$.
Turning back to the Cox-Ingersoll-Ross process  this gives us strong convergence of the Euler schemes \eqref{euler_2}, \eqref{euler_3} and \eqref{euler_1} under the assumption
$ 2 \kappa \lambda \geq \theta^2.$
Note that for the Euler schemes \eqref{euler_2} and \eqref{euler_3}
strong convergence without a restriction on the parameter has been shown in \cite{delbaen} and \cite{hm} using a Yamada function technique.   This technique has also been applied by 
Gy\"ongy \& R\'asonyi in \cite{Gyoengy_Rasonyi} to obtain the following result: 

\smallskip

\begin{theorem}\label{euler_hoelder} Let $a_1,a_2,b: \mathbb{R} \rightarrow \mathbb{R}$. Consider the one-dimensional SDE
$$dX_t = (a_1(X_t) +a_2(X_t))\,dt + b(X_t)\, dW_t, \quad t \in [0,T], \qquad X_0=x_0 \in \mathbb{R}$$  and let $\X^n$ be the corresponding Euler scheme with stepsize $\Delta=T/n$. 
Moreover, let $a_2$ be monotonically decreasing and assume that there exists constants $\alpha \in [0,1/2]$, $\beta \in (0,1]$ and $C>0$ such that
\begin{align*} |a_1(x)-a_1(y)| \leq C \cdot |x-y|, \qquad  |a_2(x)-a_2(y)| \leq C \cdot |x-y|^{\beta},
\end{align*} \begin{align*} |b(x)-b(y)| \leq C \cdot |x-y|^{\frac{1}{2} + \alpha}
\end{align*}
for all $x,y \in \mathbb{R}$. Then, for all $p \in \mathbb{N}$, there exist constants  $K_p^{\alpha,\beta} >0$  such that
$$ \mathbf{E}  \max_{k=0, \ldots, n} |X_{t_k}- \X_{t_k}^n |^p  \leq  \left \{ \begin{array}{lcl}K_p^{0,\beta} \cdot \frac{1}{\log(n)} & \textrm{for}&  \alpha=0  \\ K_p^{\alpha,\beta} \cdot \left( \frac{1}{n^{\alpha}} + \frac{1}{n^{\beta/2}} \right) & \textrm{for}&  \alpha \in (0,1/2)\\
K_p^{\alpha,\beta} \cdot \left( \frac{1}{n^{p/2}} + \frac{1}{n^{\beta p/2}} \right) & \textrm{for}&  \alpha =1/2\\
 \end{array} \right.$$
\end{theorem}

\medskip

This result can be applied to the CEV model 
$$ dX_t=\mu X_t \, dt + \sigma X_t^{\gamma} \, dW_t,$$
if the mapping $[0, \infty) \ni x \mapsto x^{\gamma} \in [0, \infty) $ is  extended to $(-\infty, 0)$, e.g. as  $(x^+)^{\gamma}$ or $|x|^{\gamma}$. Theorem \ref{euler_hoelder} then yields  strong convergence of the corresponding Euler schemes.

\smallskip
\begin{example} {\rm  Whether the convergence rates predicted from Theorem \ref{euler_hoelder}
are sharp  for the CEV model   remains an open problem. The following simulation study suggests that the Euler scheme has strong convergence order 1/2, at least for some parameter constellations. To  better preserve the positivity of the CEV process,  the Euler scheme is applied  to the SDE
$$ dX_t= \mu |X_t| \, dt + \sigma (X_t^+)^{\gamma} \, dW_t$$
   which still fulfills the assumptions of Theorem \ref{euler_hoelder} with $a_2=0$, i.e. $\beta=1$, and $\alpha=\gamma-1/2$. Its    solution  coincides with the CEV process.

Figure \ref{figure_cev_conv-rate} shows the empirical root mean square maximum error in the discretization points versus the step size for the  parameters
\begin{align*}
\textrm{\,Set I:} &\qquad \mu=0.1, \quad \sigma =0.3,  \quad \gamma=0.75, \quad T=1, \quad x_0=0.2\\
\textrm{Set II:} &  \qquad \mu=0.2, \quad  \sigma =0.5, \quad \gamma =0.55, \quad T=1, \quad x_0=0.5
\end{align*}
The empirical  mean square maximum error in the discretization points is estimated
by 
$$  \left( \frac{1}{N} \sum_{i=1}^N    \max_{k=0, \ldots, n} |X_{t_k}^{*,(i)}- \X_{t_k}^{n,(i)} |^2\right)^{1/2}
$$
with $N=5 \cdot 10^4$. Here $X^{*}$ is the numerical reference solution obtained by using the same Euler scheme with very small step size and
$X^{*,(i)}, \X^{n,(i)}$ are independent copies of  $X^{*}, \X^{n}$. For both sets of parameter a good accordance with the convergence order $1/2$ is obtained.
(The dots in the figure indicate convergence order $1/2$). 
\begin{figure}[htp]
 \leavevmode  \epsfig{figure=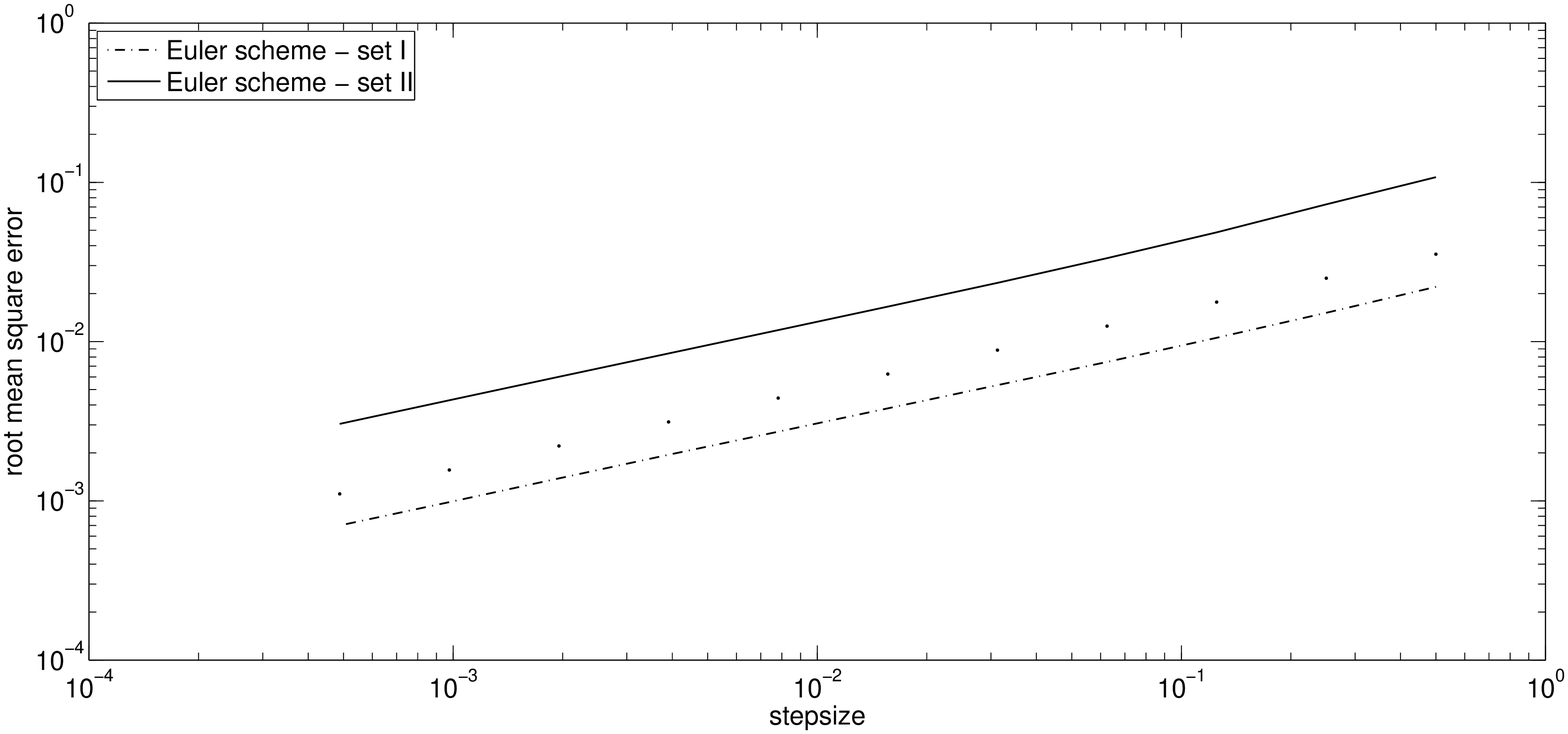, width=13cm,
height=8cm}  
\caption{\label{figure_cev_conv-rate} Root mean square error of the Euler scheme vs. step size for the CEV process for the parameter sets I  and II}
\end{figure}
A regression of the numerical data yields moreover the empirical convergence order 0.493923 for set I, respectively 0.509903 for set II.
\begin{flushright} $\diamond$ \end{flushright}}

\end{example}
\smallskip

But do  Theorems \ref{euler_hms_thm} and \ref{euler_hoelder}   have any consequences for the other examples? Unfortunately not: for the Heston, Ait-Sahalia and 3/2-models, no linear growth 
condition is satisfied. Even worse,   for the 
 Ait-Sahalia model and the 3/2-model the moments of the Euler scheme explode! In the case of the 3/2-model the latter can be deduced from the following Theorem, which was obtained by 
 Hutzenthaler, Jentzen \& Kloeden in 
  \cite{HMS1}.

 \smallskip

 \begin{theorem} \label{ref_euler_mom_exp} Let $a,b: \R \rightarrow \R$ and assume that the one-dimensional SDE \begin{equation*}
dX_t = a(X_t)\,dt + b(X_t)\,dW_t, \quad t \in [0,T],
\qquad
X_0 = x_0 \in \mathbb{R}
\end{equation*}
has a unique strong solution with $$\sup_{t \in [0,T]} \mathbb{E} | X_t |^p
< \infty $$ for one
$p \in [1,\infty)$.
Moreover, let $b(x_0) \neq 0$ and let
$ C \geq 1 $,
$ \beta > \alpha > 1 $ be constants
such that
\begin{equation*}  
 \max\!\big(\left| a(x) \right|,
 \left| b(x) \right|
 \big) \geq 
 \frac{1}{C} \cdot 
   \left| x \right|^{\beta}
 \quad
  \text{and}
 \quad
 \min\!\big(\left| a(x) \right|,
 \left| b(x) \right|
 \big) \leq 
 C \cdot  | x |^{\alpha}
\end{equation*}
for all $ |x| \geq C $. 
Then, the corresponding Euler scheme $\X^n$ with stepsize $\Delta=T/n$ satisfies
\begin{equation}   
\label{eq:euler_divergence}
  \lim_{n \rightarrow \infty}
  \mathbb{E}
    |X_T - \X^n_{T}|^p= \infty
  \quad
  \text{and}
  \quad
  \lim_{n \rightarrow \infty}
  \big| 
    \mathbb{E}|X_T|^p
    - 
    \mathbb{E}| \X^n_T|^p 
  \big| 
  = \infty .
\end{equation}
\end{theorem}

\smallskip

 In the case of the 3/2-model, which has finite moments up to order $p < 2+ \frac{2c_1}{c_3^2}$,  the coefficients are  
 $$ a(x)=- c_2 x^2 + c_1 c_2 x, \qquad b(x)= c_3(x^{+})^{3/2}, \qquad x \in \mathbb{R},$$
 so the assumptions of the above Theorem are satisfied for $\alpha=3/2$, $\beta=2$ and $C$ sufficiently large.
 
Concerning the Ait-Sahalia model, the moments of the Euler scheme  already explode in the second step. Here the first step  of the Euler scheme  has a Gaussian distribution with mean $x_0 + (\alpha_{-1}x_0^{-1} - \alpha_0 + \alpha_1 x_0 - \alpha_2 x_0^{r}) \Delta $ 
and variance $\alpha_3^2 x_0^{2 \rho}\Delta$.  The
inverse of the first step must be computed  for the second step of the Euler scheme, so the moments of the second step are infinite, since inverse moments of a Gaussian random variable do not exist.
 
 \smallskip
 
Why the  moments of the Euler scheme diverge for superlinearly growing coefficients -- even without a singularity -- can  be nicely illustrated by considering the 
SDE
\begin{equation}\label{exsde2}
d X_t =
- X_t^3\,dt + \sigma dW_t,
\qquad
X_0 = x_0
\end{equation} with $\sigma \geq 0$
for which the Euler scheme reads as
\begin{equation}\label{exeuler2}
\X^n_{t_{k+1}}
= \X^n_{t_k} \left(1 -  |\X^n_{t_k}|^2 \Delta \right)
+ \sigma \Delta_k W.
\end{equation}

In the deterministic case, i.e., 
\eqref{exsde2} and \eqref{exeuler2} with $\sigma=0$, 
the Euler approximation of the deterministic equation
is known to
be unstable if the initial value is large
(see e.g. Chapter 6 in \cite{db02}).
For example, if $x_0=n$, $T=1$ and $\Delta = n^{-1}$ then
\begin{equation*}
\X_{t_1}^n= n \left(1- \frac{n^2}{n}\right) 
\approx-n^2
\end{equation*}
and therefore
$$
\X_{t_2}^n = \X_{t_1}^n \left(1 -  |\X^n_{t_1}|^2 \Delta \right)
\approx n^5. 
$$
Iterating this further, one obtains
$$ \left| \X^n_{t_k} \right| 
\gtrapprox n^{ \left( 2^k \right) } $$
for  $k=0,1,\dots,n$.
Thus, $\X_{t_n}^n$ grows double-exponentially
fast in $n$. In the presence of noise ($\sigma>0$) there is an
exponentially small event that the Brownian motion leaves the
interval $[-2 n,2 n]$ and on this event the
 approximations grow double-exponentially fast due to the deterministic dynamics.
Consequently this double-exponentially  growth can not be  compensated by the exponentially small probability of this event, which leads to the moment explosion of the 
Euler approximation.

\smallskip

\begin{example}
{\rm   That rare events lead to the explosion of the moments of the Euler scheme can be also seen  from the following numerical example.
Consider the volatility process in the 3/2-model 
$$ dV_t =  c_1 V_t( c_2 -  V_t) \, dt +  c_3 V_t^{3/2}  \, dW_t,  \quad V_0=v_0>  0$$
with 
$$ c_1=1.2, \qquad  c_2=0.8, \qquad  c_3=1, \qquad T=4, \qquad v_0= 0.5$$
and try to compute 
$$ {\mathbf E} |X_T|=  0.566217
$$
using the standard  Monte Carlo estimator 
$$  \frac{1}{N} \sum_{i=1}^{N} |\X_{T}^{n,(i)}| $$
where $\X_{T}^{n,(1)}, \ldots, \X_{T}^{n,(N)}$ are iid copies of $\X_{T}^{n}$.
The exact value for ${\mathbf E} |X_T|$  is computed using the inverse moments of the CIR process, see e.g. \cite{HK}.
 While for a moderate number of repetitions  the estimator seems to converge for small step sizes (and the 'Inf'-outputs seem to be some numerical instabilities due to the large step sizes), the estimator explodes even for small step sizes when increasing the number of repetitions -- as predicted by Theorem \ref{ref_euler_mom_exp}. 
 Despite of this the Euler scheme for this SDE  converges pathwise with rate $1/2-\varepsilon$ due to Theorem \ref{taylor_mod}.
\medskip

 \begin{center}
\begin{tabular}{c||c|c|c|c|c|c|} 
Repetitions $N$ /  stepsize  $\Delta$ & $2^0$  &  $ 2^{-2}$ & $ 2^{-4 }$ & $2^{-6}$ &  $ 2^{-8}$ & $2^{-10}$  \\ \hline \hline
 $10^3$  &  6.327232 & Inf  & Inf  &  0.550185  &  0.553499  &  0.555069
 \\ \hline
 $10^4$  &   6.894698 &  Inf &  Inf & Inf &  0.562716 &  0.563352  \\ \hline
 $10^5$  &   7.430606 & Inf  &  Inf & Inf &  0.566218 & 0.567106
 \\ \hline
 $10^6$  &   7.227379 & Inf & Inf &  Inf &  Inf &  0.565750
 \\ \hline
 $10^7$  &   7.279187 & Inf & Inf &  Inf &  Inf &  Inf
\end{tabular}
\end{center}

\medskip

 A similar moment explosion arises if a Multi-level Monte Carlo method is used to estimate $ {\mathbf E} |X_T|$. This is shown, also for more general SDEs, in  \cite{MLMC_expl}.

 \begin{flushright} $\diamond$ \end{flushright}}
\end{example} 
\smallskip

However, in some cases using the Euler scheme one still obtains a convergent Monte Carlo estimator for functionals of the type $\mathbf{E} \phi(X_T)$. 
The standard Euler-based estimator for the latter quantity 
is 
\begin{align} 
 \frac{1}{N} \sum_{i=1}^{N} \phi \big(\X_{T}^{n,(i)} \big). \label{MC_st}
\end{align}
In the classical case, i.e.
if $a,b,\phi \in C^{4}(\R; \R)$ with at most polynomially growing derivatives and $a$, $b$ globally Lipschitz, one has
$$ \mathbf{E} 
  \left|
    \frac{1}{N} 
      \sum_{i=1}^{N} \phi \big(\X_{T}^{n,(i)} \big) - \mathbf{E} \phi(X_T) 
  \right|^2 
  \leq  K_{Bias} \cdot \frac{1}{n^2} + K_{MC} \cdot \frac{1}{N},
$$ see e.g. \cite{KP}. The first term on the right hand side corresponds to the squared bias of the Euler scheme, while the second term corresponds to the variance of the Monte Carlo simulation. 
It is thus optimal to choose $N=n^2$ for balancing both  terms with respect to the computational cost (number of arithmetic operations, function evaluations and   random numbers used), see \cite{DG}.  The corresponding 
Monte Carlo estimator has then convergence order $1/3$ in terms of the computational cost.

\smallskip

Hutzenthaler \& Jentzen could show in \cite{HuJe} that if the global Lipschitz assumption on the drift-coefficient is weakened to
 \begin{equation}\label{snew.onesided}
  (x-y)  ( a (x) - a (y) ) \leq 
  L\left(x-y\right)^2,
  \quad \, x,y\in\mathbb{R} 
\end{equation}
for some $L>0$, then one still has
\begin{equation*}  
  \left|
    \frac{1}{N^2} 
      \sum_{i=1}^{N^2} \phi (\X_{T}^{N,(i)} ) -  \mathbf{E} \phi(X_T)
  \right|
  \leq  \eta_{\varepsilon} \cdot N^{-(1-\varepsilon)}
\end{equation*}
almost surely
for all $\varepsilon >0$ and almost-surely finite and non-negative random variables $\eta_{\varepsilon}$.

\smallskip

Weak approximation under non-standard assumptions is also studied by Milstein \& Tretyakov in  \cite{Miltret}. In their approach, simulations which leave a ball with sufficiently large radius are discarded. 
In the context of the Euler scheme with equidistant stepsize  this estimator reads as
$$ \frac{1}{N} 
      \sum_{i=1}^{N} \phi \big (\X_{T}^{n,(i)} \big) \cdot \mathbf{1}_{ \{ \sup_{k=0, \ldots, n} |X_{t_k}^{n,(i)}| \leq  R \} }. $$
For coefficients $a,b$ and functions $\phi$ satisfying a Lyapunov-type condition still a convergent Monte Carlo estimator is obtained, when matching the discarding radius $R$ appropriately to the number of repetitions $N$ and the stepsize of the discretization $n$.

\smallskip

Condition \eqref{snew.onesided} on the drift coefficient is the so-called one-sided Lipschitz condition. This condition is also very useful to obtain strong convergence results for implicit Euler methods and tamed Euler schemes, which   will be explained in the next section.
Very recently a unifying framework for the analysis of Euler-type methods has been provided in \cite{HJ_uni}.

\bigskip
\bigskip

\section{Strong convergence of implicit and tamed Euler schemes}

The condition in Theorem  \ref{euler_hms_thm}  for the strong convergence of the Euler scheme which is usually  difficult to verify is the finiteness of its moments, i.e.
$$ \sup_{n \in \mathbb{N}} \, \mathbf{E} \max_{k=0, \ldots, n} |\overline{X}_{t_k}^n|^p < \infty $$
for some $p>2$. Moreover, this condition may even fail to hold for specific equations, see Theorem \ref{ref_euler_mom_exp}.
 However, both problems can be overcome in some situations if appropriate drift-implicit Euler schemes are used. The split-step backward Euler scheme
is defined as
\begin{align}\label{sses}
 X^{*}_{t_{k}} &= \X_{t_k} +  a(X^*_{t_{k}})\Delta, \qquad  \quad
\overline{X}_{t_{k+1}} = X^{*}_{t_{k}} + \sum_{j=1}^m b_j(X^*_{t_{k}}) \Delta_k W^{(j)}
\end{align}
for $k=0,1, \ldots $ with $ \overline{X}_{0}=x_0$, while the backward  or drift-implicit Euler scheme reads as
\begin{align}\label{bes}
 \overline{X}_{t_{k+1}} = \overline{X}_{t_{k}} + a(\overline{X}_{t_{k+1}}) \Delta+  \sum_{j=1}^m b_j(\overline{X}_{t_{k}}) \Delta_k W^{(j)}. 
 \end{align}
Both schemes are defined via an implicit equation, whose solvability relies  on the properties of the drift-coefficient $a$. The following result has been obtained by 
Higham, Mao \& Stuart in \cite{HMS}.

\smallskip

\begin{theorem}\label{euler_sses_thm} Let $a, b_j \in C^{1}(\mathbb{R}^d; \mathbb{R}^d)$, $j=1, \ldots, m$, and assume that there exist constants $L_1,L_2>0$ such that
\begin{align*}\langle x-y,a(x)-a(y) \rangle & \leq L_1 \cdot |x-y|^2, \qquad x,y \in \mathbb{R}^d, \\
 \sum_{j=1}^m |b_j(x)-b_j(y)|^2 & \leq L_2  \cdot |x-y|^2, \qquad        x,y \in \mathbb{R}^d.    
\end{align*} Then, the split-step backward Euler scheme given by \eqref{sses} with stepsize $\Delta =T/n$ is well defined for $\Delta <  \Delta_*:=1/ \max \{ 1+ 2L_1, 4L_2 \} $ and satisfies
$$ \lim_{n \rightarrow \infty}  \mathbf{E}  \max_{k=0, \ldots, n} |X_{t_k}- \X_{t_k}^n |^2=0. $$
\end{theorem}

\smallskip

The conditions on the coefficients imply that the SDE has bounded moments of any order, and  also allow  one to show that  the split-step  Euler method has moments of any order. 
The implicitness of the method is crucial for the latter. 
 Furthermore, the split-step Euler method coincides with the explicit Euler method for the perturbed SDE
\begin{align} dX_t^{\Delta}= a(h_{\Delta}(X_t^{\Delta}))\,dt  + \sum_{j=1}^m b_j(h_{\Delta}(X_t^{\Delta})) \, dW^{(j)}(t), \qquad X_0^{\Delta}=x_0.  
\label{pertub_sde} \end{align}
Here the function $h_{\Delta}: \mathbb{R}^d \rightarrow \mathbb{R}^d$ is defined as the unique solution of the equation
$$ h_{\Delta}(x) = x +  a(h_{\Delta}(x)) \Delta, \qquad x \in \mathbb{R}^d$$
 with $\Delta < \Delta_*$. Since $h_{\Delta}$ converges to the identity for $\Delta \rightarrow 0$ this perturbed SDE is close to original SDE. To establish  Theorem \ref{euler_sses_thm}, it thus 
 remains to show that the split-step backward Euler scheme is close to \eqref{pertub_sde}, which can be done along the lines of the proof of Theorem \ref{euler_hms_thm}.
 
 \smallskip

If the drift-coefficient is additionally also polynomially Lipschitz, then      the standard strong convergence rate $1/2$ can even be recovered.

\begin{theorem}\label{euler_sses_bes_thm} Let the assumptions of Theorem \ref{euler_sses_thm} hold and assume additionally that there exist $C,q>0$ such that
$$ |a(x)-a(y)| \leq C \cdot (1+|x|^q + |y|^q) \cdot |x-y|, \qquad x,y \in \mathbb{R}^d.$$   Then, the split-step backward Euler scheme given by \eqref{sses} and the backward Euler scheme given by \eqref{bes} are well defined for  $\Delta <  \Delta^*$ and have strong convergence order $1/2$, i.e. for both schemes there exists a constant $K>0$ such that
$$   \mathbf{E}  \max_{k=0, \ldots, n} |X_{t_k}- \X_{t_k}^n |^2 \leq K \cdot n^{-1}. $$
 \end{theorem}

\smallskip
As pointed out above, in each step of  both schemes an implicit equation has to be solved. 
If the function $h_{\Delta}$ is not known explicitly, this has to be done numerically and
may be time-consuming. Solving implicit equations can be avoided by using the so-called tamed Euler method, which has been proposed
by Hutzenthaler, Jentzen \&  Kloeden in \cite{HMS2}:
\begin{align}
\label{tamed_euler}
\overline{X}_{t_{k+1}} = \overline{X}_{t_{k}}
  + \frac{1}{ 1 + | a(\overline{X}_{t_{k}})| \Delta} a(\overline{X}_{t_{k}}) \Delta +  \sum_{j=1}^m b_j(\overline{X}_{t_{k}}) \Delta_k W^{(j)}.
 \end{align}

Here the drift-term is ``tamed" by the factor
$  \frac{1}{1+  | a(\overline{X}_{t_k}) | \Delta} $ 
in the $k$-th step, which prevents a possible explosion of the scheme. 

\smallskip

\begin{theorem}\label{euler_tamed_thm} Let the assumptions of Theorem \ref{euler_sses_bes_thm} hold.
Then, there exists a constant $K>0$ such that the tamed Euler scheme given by \eqref{tamed_euler}  satisfies
$$   \mathbf{E}  \max_{k=0, \ldots, n} |X_{t_k}- \X_{t_k}^n |^2 \leq K \cdot n^{-1}. $$
\end{theorem}

\smallskip

Here,  the difficulty is again to control the moments of the approximation scheme. For this appropriate processes are used that dominate the tamed Euler scheme on subevents whose probabilities
 converge sufficiently fast to one.
\medskip

The Theorems given so far in this section  require the diffusion coefficient to be globally Lipschitz, which is often not fulfilled in SDEs arising from mathematical finance.
 However, the backward Euler method can be also  successfully applied to the Ait-Sahalia interest rate model
\begin{align}\label{as_sde}
dX_t =  \big(\alpha_{-1} X_t^{-1} - \alpha_0 + \alpha_1 X_t - \alpha_2X_t^r \big)dt +\sigma X_t^{\rho} dW_t
\end{align}
where  $ \alpha_{i}, \sigma > 0$, $i=-1, \ldots, 2$ and $ r, \rho > 1$. 
The following result has been obtained by Szpruch et al.  in \cite{SMHJ}:

\smallskip

\begin{theorem}
Consider the SDE \eqref{as_sde}
and assume that $$r+1 >2 \rho.$$ Then the corresponding backward Euler method \eqref{bes} with stepsize $\Delta=T/n$  is well defined if  $\Delta \leq 1/\alpha_1$,  and  
$$ \lim_{n \rightarrow \infty}  \mathbf{E}  \max_{k=0, \ldots, n} |X_{t_k}- \X_{t_k}^n |^2=0. $$
\end{theorem}

\smallskip
Here the drift coefficient is still one-sided Lipschitz on the domain of the SDE, i.e.  
$$  (x-y) (a(x)-a(y)) \leq \alpha_1 |x-y|^2, \qquad x,y>0,
$$
and, moreover, $-a$ is  coercive on $(0, \infty)$, i.e.
$$ \lim_{x \rightarrow 0} a(x)= \infty \qquad \quad \lim_{x \rightarrow \infty} a(x)= -\infty.$$
These two properties ensure that the drift-implicit Euler scheme  for \eqref{as_sde} is well-defined and, in particular,  takes  only strictly positive values.

\smallskip
The drift coefficient in the volatility process 
$$ dV_t =  c_1 V_t( c_2 -  V_t) \, dt +  c_3 V_t^{3/2}  \, dW_t,  \quad V_0=v_0>  0$$
in the 3/2-model is  also one-sided Lipschitz on $(0, \infty)$. It does not,  however,  satisfy the coercivity assumption. Consequently, the drift-implicit Euler scheme cannot be applied here, since the implicit equation may not be solvable.
Note that very recently, Higham et al. introduced in \cite{mils_szp} a double-implicit Milstein scheme, which is strongly convergent for the
3/2-model and similar SDEs.

\bigskip
\bigskip

\section{Strong Convergence  Rates for the approximation of the Cox-Ingersoll-Ross process and the Heston model}

Strong convergence rates for the approximation of the CIR process
$$dX_t = \kappa (\lambda - X_t) \, dt + \theta \sqrt{X_t} \, dW_t, \quad t \in [0,T], \qquad X_0=x_0>0 $$
with $\kappa, \lambda, \theta >0$ have been a long standing open problem, even
in  the regime where the CIR process does not hit zero, i.e. when 
 $2 \kappa \lambda \geq \theta^2$.
 
The first non-logarithmic  rates were derived by Berkaoui, Bossy \& Diop for the symmetrized Euler scheme \eqref{euler_1}, i.e.
\begin{align*}
\X_{t_{k+1}} &= \Big | \X_{t_k} +  \kappa( \lambda
  - \X_{t_k} )  \Delta  + \theta  \sqrt{
     \X_{t_k}} \, \Delta_k W \Big| . 
\end{align*} 
They  showed in \cite{diop_2} that 
$$\EX \max_{k=0, \ldots, n} |X_{t_k} - \X_{t_k}|^{2p} \leq C_p\cdot \Delta^p$$
under the assumption
\begin{align*}   \frac{2 \kappa \lambda}{\theta^2} >  1+   \sqrt{8} \, \max \left \{  \frac{\sqrt{\kappa}}{\theta} \sqrt{16p-1},  16p-2 \right \}, \end{align*}
where the constant $C_p>0$ depends only on $p, \kappa,
  \lambda, \theta, x_0$ and $T$. Strong convergence rates for a drift-implicit Euler-type scheme  were  recently obtained under mild assumptions by Dereich, Neuenkirch \& Szpruch in \cite{dns}.
 Their key tool is the use of the  Lamperti-transformation:  by the  It\^o formula, the  transformed process $Y_t = \sqrt{X}_t$ satisfies the SDE
 \begin{align} \label{sqr_X}
dY_t = \frac{\alpha}{Y_t} \,dt + \beta Y_t \, dt  + \gamma \, dW_t, \quad t \geq 0,
	\qquad   Y_0=\sqrt{x_0}  \end{align}
with
$$ \alpha = \frac{4 \kappa \lambda - \theta^2}{8}, \qquad \beta = - \frac{\kappa}{2}, \qquad \gamma = \frac{\theta}{2}.$$
 At first glance this transformation
does not help at all, since the drift coefficient of the arising SDE is singular. However,
$$ a(x)=   \frac{\alpha}{x} + \beta x, \qquad x > 0, $$ satisfies 
for  $\alpha >0$, $\beta \in \mathbb{R}$ the restricted one-sided  Lipschitz condition
\begin{align*} 
 (x-y)(a(x)-a(y)) \leq \beta (x-y)^2, \qquad x,y > 0
 \end{align*}
The  drift-implicit Euler method with stepsize $\Delta>0$ in this case is  \begin{align*} 
 \Y_{t_{k+1}} = \Y_{t_{k}} +  \left( \frac{\alpha}{\Y_{t_{k+1}}} + \beta \Y_{t_{k+1}} \right) \Delta + \gamma \Delta_k W , \qquad k=0,1, \ldots \end{align*}  
with $\overline{Y}_0 = \sqrt{x_0}$, which   has the explicit solution
\begin{align*} 
\Y_{t_{k+1}}=  \frac{\Y_{t_{k}} + \gamma \Delta_k W}{2(1- \beta \Delta)} +  \sqrt{ \frac{(\Y_{t_{k}} + \gamma \Delta_k W)^2}{4(1- \beta \Delta)^2} + \frac{\alpha \Delta}{1- \beta \Delta} }.
\end{align*}
Setting \begin{align} \label{sqr_euler} \X_{t_k} = \Y^2_{t_k}, \qquad k=0,1, \ldots  , \end{align}
gives a positivity preserving approximation of the CIR process,  which  is called drift-implicit square-root Euler method.
This scheme had already been proposed in \cite{alfonsi}, but without  a convergence analysis.
Piecewise linear interpolation, i.e.
\begin{align*} 
\overline{X}_t=\frac{t_{k+1}-t}{\Delta}\X_{t_k}+ \frac{t -t_k}{\Delta}\X_{t_{k+1}}, \qquad t \in [t_k,t_{k+1}], \end{align*}
 gives a global approximation $(\overline{X}_t)_{t \in [0,T]}$ of the CIR process on $[0,T]$.
The main result of \cite{dns} is:
\smallskip
\begin{theorem} \label{mainthm_sqr_euler}
Let $2\kappa \lambda > \theta^2$, $x_0>0$ and $T>0$. Then, for all $$ 1 \leq p < \frac{2 \kappa \lambda}{\theta^2} $$ there exists a constant $K_p >0$   such that
\begin{align*}\left( \EX \max_{t \in [0,T]} |X_{t} - \overline{X}_t |^p  \right)^{1/p} \leq K_p \cdot \sqrt{\log(\Delta)|} \cdot \sqrt{\Delta}, \end{align*}
for all  $\Delta \in (0, 1/2]$.
\end{theorem} 

\smallskip

The restriction on $p$ arises in the proof of the convergence rate when controlling the inverse $p$-th moments of the CIR process, which are infinite for $p \geq 2 \kappa \lambda /\theta^2$. For further details, see \cite{dns}.
Note that  for SDEs with Lipschitz coefficients the  convergence rate $\sqrt{| \log(\Delta)|} \cdot \sqrt{\Delta}$   is best possible  with respect to the above global error criterion, see \cite{tmg}. So the convergence rate given in Theorem \ref{mainthm_sqr_euler}  matches the rate that is optimal under standard assumptions.

\smallskip
Other approximation schemes for the strong approximation of the CIR process can be found in
\cite{alfonsi,jk,hal}. Among them is the   drift-implicit Milstein scheme
  \begin{align*} 
\Z_{t_{k+1}} =\Z_{t_{k}} & +    \kappa( \lambda -\Z_{t_{k+1}}) \Delta +            \theta \sqrt{\Z_{t_{k}}} \Delta_k W + \frac{\theta^2}{4} \big( (\Delta_k W)^2 - \Delta \big) 
\end{align*} 
with $Z_0=x_0$, see \cite{jk}.  It can  be rewritten as
\begin{align} \label{dimp_niceform}
\Z_{t_{k+1}} = \frac{1}{1+\kappa \Delta } \left( \sqrt{\Z_{t_k}} +   \frac{\theta}{2} \Delta_k W \right)^2 +   \frac{1}{1+\kappa \Delta } \left( \kappa \lambda - \frac{\theta^2}{4} \right) \Delta,       
\end{align} 
so this scheme preserves the positivity of the CIR process
if $4\kappa \lambda \geq \theta^2$. It coincides up to a term of second order  with the drift-implicit square-root Euler method, since the latter can be written as
\begin{align*} 
 \X_{t_{k+1}}  = & \frac{1}{1+\kappa \Delta } \left( \sqrt{\X_{t_{k}}} +   \frac{\theta}{2} \Delta_k W \right)^2 +   \frac{1}{1+\kappa \Delta } \left( \kappa \lambda - \frac{\theta^2}{4} \right) \Delta \\
 & \qquad -  \frac{1}{1+\kappa \Delta} \left(  \frac{4 \kappa \lambda - \theta^2}{8 \sqrt{\X_{t_{k+1}}}} - \frac{\kappa}{2} \sqrt{\X_{t_{k+1}}} \right)^2 \Delta^2.
\end{align*}
Moreover the drift-implicit Milstein scheme
 dominates  the drift-implicit square-root Euler method:
\begin{lemma}\label{dom_Z} Let $2\kappa \lambda > \theta^2$, $x_0>0$ and $T>0$.
Then 
$$ \mathbf{P} (\Z_{t_k} \geq \X_{t_k}, \,\, k=0,1, \ldots ) =1. $$
\end{lemma}

\begin{proof}
The numerical flow for the drift-implicit Milstein scheme is given by
  \begin{align*} 
\varphi_{\overline{Z}}(x;k,\Delta)= \frac{1}{1+\kappa \Delta } \left( \sqrt{x} +   \frac{\theta}{2} \Delta_k W \right)^2 +   \frac{1}{1+\kappa \Delta }\left( \kappa \lambda - \frac{\theta^2}{4} \right) \Delta   
\end{align*} 
and  for the drift-implicit square-root Euler method   it  satisfies
 \begin{align*} 
 \varphi_{\overline{X}}(x;k,\Delta) & + \frac{1}{1+\kappa \Delta} \left(  \frac{4 \kappa \lambda - \theta^2}{8 \sqrt{\varphi_{\overline{X}}(x;k,\Delta) }} - \frac{\kappa}{2} \sqrt{ \varphi_{\overline{X}}(x;k,\Delta)} \right)^2 \Delta^2  \\ &  \qquad = \frac{1}{1+\kappa \Delta } \left( \sqrt{x} +   \frac{\theta}{2} \Delta_k W \right)^2 +   \frac{1}{1+\kappa \Delta } \left( \kappa \lambda - \frac{\theta^2}{4} \right) \Delta.
\end{align*} 
 From \cite{alfonsi} it is known that $\varphi_{\overline{X}}$ is monotone, i.e.
$$  \varphi_{\overline{X}}(x_1;k,\Delta) \geq   \varphi_{\overline{X}}(x_2;k,\Delta) $$ 
for $x_1 \geq x_2$. Thus it remains to show that
$$  \varphi_{\overline{Z}}(x;k,\Delta) \geq   \varphi_{\overline{X}}(x;k,\Delta) $$
for arbitrary $\Delta>0$, $k=0, 1, \ldots$, $x >0$. However,  this follows  directly  by  comparing 
both flows. 
\end{proof}

\smallskip

The above property allows one  to show the strong convergence  of the drift-implicit Milstein scheme, which seems  not to have  been established yet in the literature.

\smallskip

\begin{proposition}\label{conv_dimp_mil} Let $2\kappa \lambda > \theta^2$, $x_0>0$ and  $T>0$.  Then 
\begin{align*}
   \lim_{n \rightarrow \infty}  \mathbf{E}  \max_{k=0, \ldots, n} |X_{t_k}- \Z^n_{t_k} |^2=0. 
 \end{align*}
\end{proposition}
\begin{proof} 

First note that the  drift-implicit square-root Euler method   can be rearranged as 
\begin{align*}
 X_{t_{k+1}} &= \varphi_{\overline{Z}}(X_{t_k};k, \Delta) - \kappa \int_{t_k}^{t_{k+1}} (X_t-X_{t_{k+1}})\,dt
 \\ & \qquad + \theta \int_{t_k}^{t_{k+1}} (\sqrt{X_t} - \sqrt{X_{t_k}}  )\,dW_t  - \frac{\theta^2}{4} ( \Delta_k W^2 - \Delta)  \end{align*}
where $\varphi_{\overline{Z}}$ is the numerical flow of the drift-implicit Milstein scheme defined
in the proof of the above Lemma.
Thus the  error $e_k=X_{t_k}-\Z_{t_k}$ satisfies the recursion
\begin{align} \label{rec_1}
 e_{k+1} &=  e_k - \kappa e_{k+1} \Delta   + \theta \Big(\sqrt{X_{t_k}}- \sqrt{\Z_{t_k}} \Big) \Delta_k W + \rho_{k+1}
 \end{align}
 with $e_0=0$, where 
$$ \rho_{k+1} = -\kappa \int_{t_k}^{t_{k+1}} (X_s-X_{t_{k+1}})\,ds + \theta \int_{t_k}^{t_{k+1}} (\sqrt{X_s} - \sqrt{X_{t_k}}  )\,dW_s.$$
Now \eqref{rec_1} gives
\begin{align*}
 e_{k+1} &=  \frac{1}{1+ \kappa \Delta} \left( e_k + \theta  \Big(\sqrt{X_{t_k}}- \sqrt{\Z_{t_k}} \Big) \Delta_k W + \rho_{k+1} \right),
 \end{align*} 
so 
$$ e_k = \sum_{\ell=0}^{k-1} \frac{\theta}{(1+ \kappa \Delta)^{k-\ell}} \Big(\sqrt{X_{t_{\ell}}}- \sqrt{\Z_{t_{\ell}}} \Big) \Delta_{\ell} W  +\sum_{\ell=0}^{k-1} \frac{1}{(1+ \kappa \Delta)^{k-\ell}} \rho_{\ell+1}.$$
Straightforward calculations using \eqref{dimp_niceform} yield
$$  \sup_{n \in \mathbb{N}} \sup_{k=0, \ldots, n} \, \mathbf{E} \Z_{t_k}  < \infty. $$ 
Then  applying the Burkholder-Davis-Gundy  inequality to the martingale
$$ M_{k}=  \sum_{\ell=0}^{k-1} (1+ \kappa \Delta)^{\ell} \Big(\sqrt{X_{t_{\ell}}}- \sqrt{\Z_{t_{\ell}}} \Big) \Delta_{\ell} W, \qquad k=0,1, \ldots $$
gives 
\begin{align} \label{crucial_est}
\mathbf{E} \sup_{k=0, \ldots, n} e_k^2   & \leq c  \sum_{\ell=0}^{n-1}  (1+ \kappa \Delta)^{2\ell} \mathbf{E} \left| \sqrt{X_{t_\ell}}- \sqrt{\Z_{t_\ell}} \right|^2 \Delta  \\ & \qquad \quad \nonumber +   c  \, \mathbf{E} \sup_{k=1, \ldots,n} \left| \sum_{\ell=0}^{k-1}  \frac{1}{(1+ \kappa \Delta)^{k-\ell}} \rho_{\ell+1} \right|^2 .\end{align}
Here and below  constants whose particular value is not important will be denoted  by $c$ regardless of their value.

It remains to   estimate   the terms on the right side of the  equation \eqref{crucial_est}. 
The previous Lemma  implies that
$$  
\mathbf{E}|\overline{Z}_{t_k} - \overline{X}_{t_k}| =   \mathbf{E}(\overline{Z}_{t_k} - \overline{X}_{t_k}),  
$$ 
so
$$ 
\mathbf{E}|\overline{Z}_{t_k} - X_{t_k}| \leq 2 \, \mathbf{E}|\overline{X}_{t_k} - X_{t_k}| + |\mathbf{E}(\overline{Z}_{t_k} - {X}_{t_k}) |.  
$$ 
Clearly, Theorem \ref{mainthm_sqr_euler} yields
$$  \max_{k=0, \ldots, n} \mathbf{E} |\overline{X}_{t_k} - X_{t_k}| \leq c \cdot \sqrt{| \log(\Delta)|} \cdot \sqrt{\Delta}.
 $$
 Moreover,  
$$ \mathbf{E}  \overline{Z}_{t_{k+1}} =    \mathbf{E} \overline{Z}_{t_k} + \kappa (\lambda - \mathbf{E}  \overline{Z}_{t_{k+1}}) \Delta,$$
which is the drift-implicit Euler approximation of 
$$ \mathbf{E} X_t = x_0 +  \int_0^t \kappa(\lambda - \mathbf{E} X_s) \,ds, \qquad t \in  [0,T],$$  at the discretization points $t_k=k \Delta$,  so
$$ \max_{k=0, \ldots, n}  |\mathbf{E} ({X}_{t_k} - \Z_{t_k})| \leq c \cdot \Delta.$$
 Hence
\begin{align*} 
  \max_{k=0, \ldots, n}  \mathbf{E} |\overline{X}_{t_k} - X_{t_k}| \leq c \cdot \sqrt{| \log(\Delta)|} \cdot \sqrt{\Delta} 
\end{align*} 
which gives
\begin{align} \sum_{k=0}^{n}  (1+ \kappa \Delta)^{2k} \mathbf{E} \left| \sqrt{X_{t_k}}- \sqrt{\Z_{t_k}} \right|^2 \Delta \leq c \cdot  \sqrt{| \log(\Delta)|} \cdot \sqrt{\Delta} \label{est_dimp_1}\end{align}
 since $|\sqrt{x}-\sqrt{y}|\leq \sqrt{|x-y|}$ for $x,y >0$ and $\sup_{n \in \mathbb{N}} \sup_{k=0, \ldots, n}  (1+ \kappa \Delta)^{2k}  < \infty$.
 
 \smallskip
 
 For the second term,  applying the Burkholder-Davies-Gundy inequality  and Jensen's inequality yield
\begin{align*}
  &  \mathbf{E} \sup_{k=1, \ldots, n} \left| \sum_{\ell=0}^{k-1} \frac{1}{(1+ \kappa \Delta)^{k-\ell}} \rho_{\ell+1} \right|^2 \\ & \qquad \leq c  \cdot \frac{1}{\Delta} \,   \sum_{k=0}^{n-1} \mathbf{E}\left| \int_{t_k}^{t_{k+1}} (X_t-X_{t_{k+1}}) \, dt \ \right|^2 + c   \sum_{k=0}^{n-1} \mathbf{E} \int_{t_k}^{t_{k+1}} \left| \sqrt{X_t}- \sqrt{X_{t_k}} \right|^2 \,dt .
  \end{align*}
 Now  
 $$ \mathbf{E} |X_t-X_s|^2 \leq c \cdot |t-s|, \qquad s,t \in [0,T], $$
so it follows that 
\begin{align}
 \mathbf{E} \sup_{k=1, \ldots, n} \left| \sum_{\ell=0}^{k-1} \frac{1}{(1+ \kappa \Delta)^{k-\ell}} \rho_{\ell+1} \right|^2 \leq c \cdot \sqrt{\Delta}, \end{align}
which completes  the proof of the proposition. 
\end{proof}

\smallskip

Alternatively, Proposition \ref{conv_dimp_mil} could have been obtained by deriving the pathwise convergence of the drift-implicit Milstein scheme and establishing the uniform integrability of the squared maximum error. Note that the above proof gives also the convergence order 1/4 up to a logarithmic term. However this rate seems to be suboptimal,  see  the following numerical example. 

\smallskip

\begin{example} {\rm The
Figures \ref{fig_cir_path_msq_1}  and  \ref{fig_cir_path_msq_2} show the empirical root mean square maximum error
 in the discretization points, i.e.  
$$  \left( \frac{1}{N} \sum_{i=1}^N    \max_{k=0, \ldots, n} |X_{t_k}^{*,(i)}- \X_{t_k}^{n,(i)} |^2\right)^{1/2},
$$ 
 versus the step size for the approximation of the CIR process.
Consider the 
 \begin{itemize}
\item[(i)]  truncated Euler scheme \eqref{euler_2}
\item[(ii)]  drift-implicit square-root Euler \eqref{sqr_euler}
\item[(iii)]   drift-implicit Milstein scheme \eqref{dimp_niceform} 
 \end{itemize}
  for the Scenarios I and II (see Example \ref{cir_scen_def}).  Scenario I satisfies the condition of Theorem \ref{mainthm_sqr_euler} and Proposition \ref{conv_dimp_mil} since $2 \kappa \lambda/\theta^2=2.011276\ldots $.This condition is violated in Scenario II where $2 \kappa \lambda /\theta^2=0.36$. In the latter scenario,  the truncation $\sqrt{x^+}$ in the definition of the schemes  (ii) and (iii) is used, since both discretization schemes may take negative values here.
 The numerical reference solution $X^*$ is computed in Scenario I using scheme (ii) with very small stepsize and in Scenario II with scheme (i) with a   very small stepsize. The number of repetitions of the Monte Carlo simulation is $N=5\cdot 10^4$.
\begin{figure}[htp]
 \leavevmode  \epsfig{figure=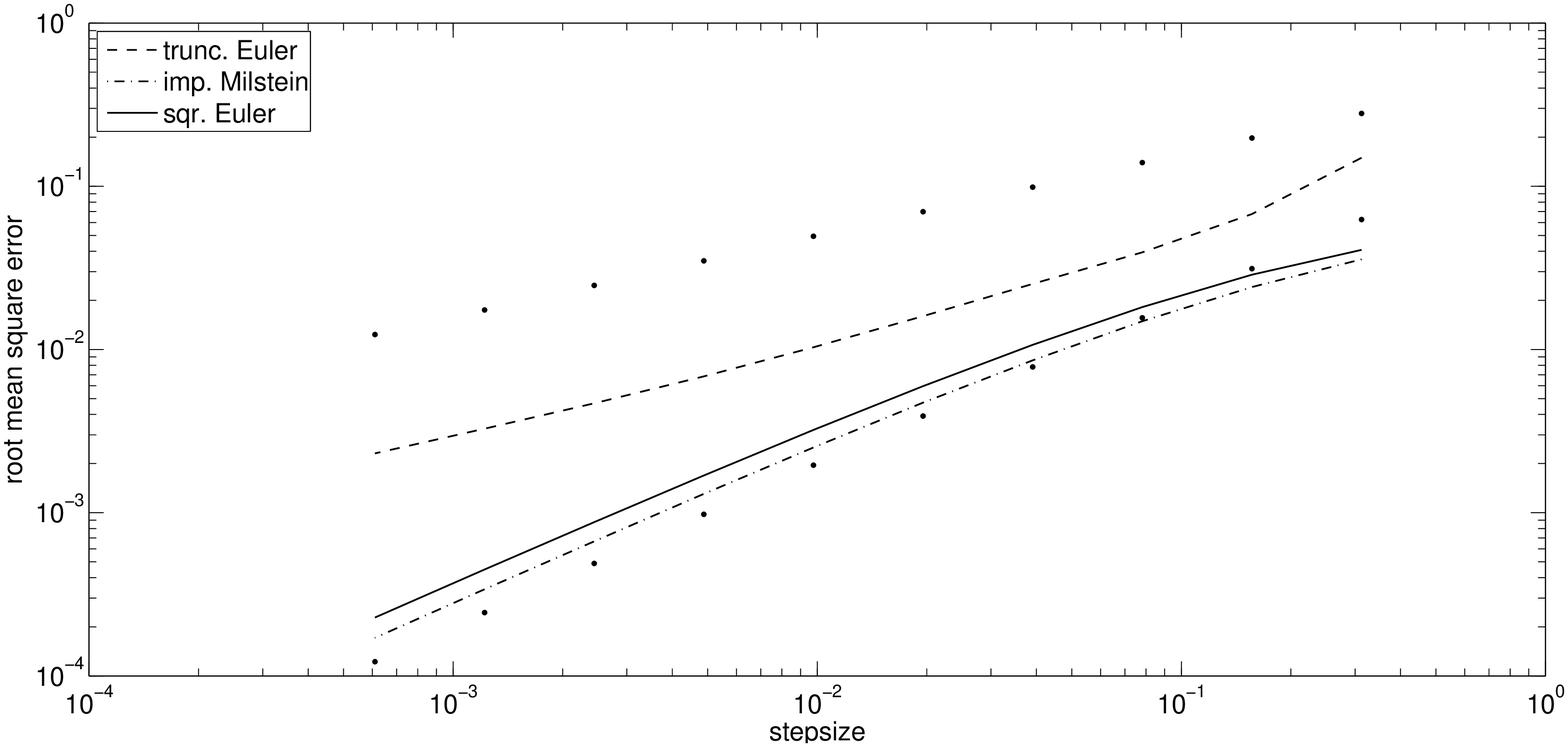, width=13cm,
height=8cm}  
\caption{\label{fig_cir_path_msq_1}Root mean square errors  vs.~step size for Scenario I of the CIR process}
\end{figure}
In the  log-log coordinates here,  the dots indicate the
convergence orders  $0.5$ and $1.0$ in Figure \ref{fig_cir_path_msq_1} and 0.25 and 0.5 in Figure  \ref{fig_cir_path_msq_2},  respectively. For Scenario I the empirical mean square error for the truncated Euler scheme  seems to decay with the order $0.5$, while the other schemes seem to have an empirical convergence order close to $1.0$. (For smooth and Lipschitz coefficients the Milstein scheme is of order one for the maximum error in the discretization points.)
\begin{figure}[htp]
 \leavevmode  \epsfig{figure=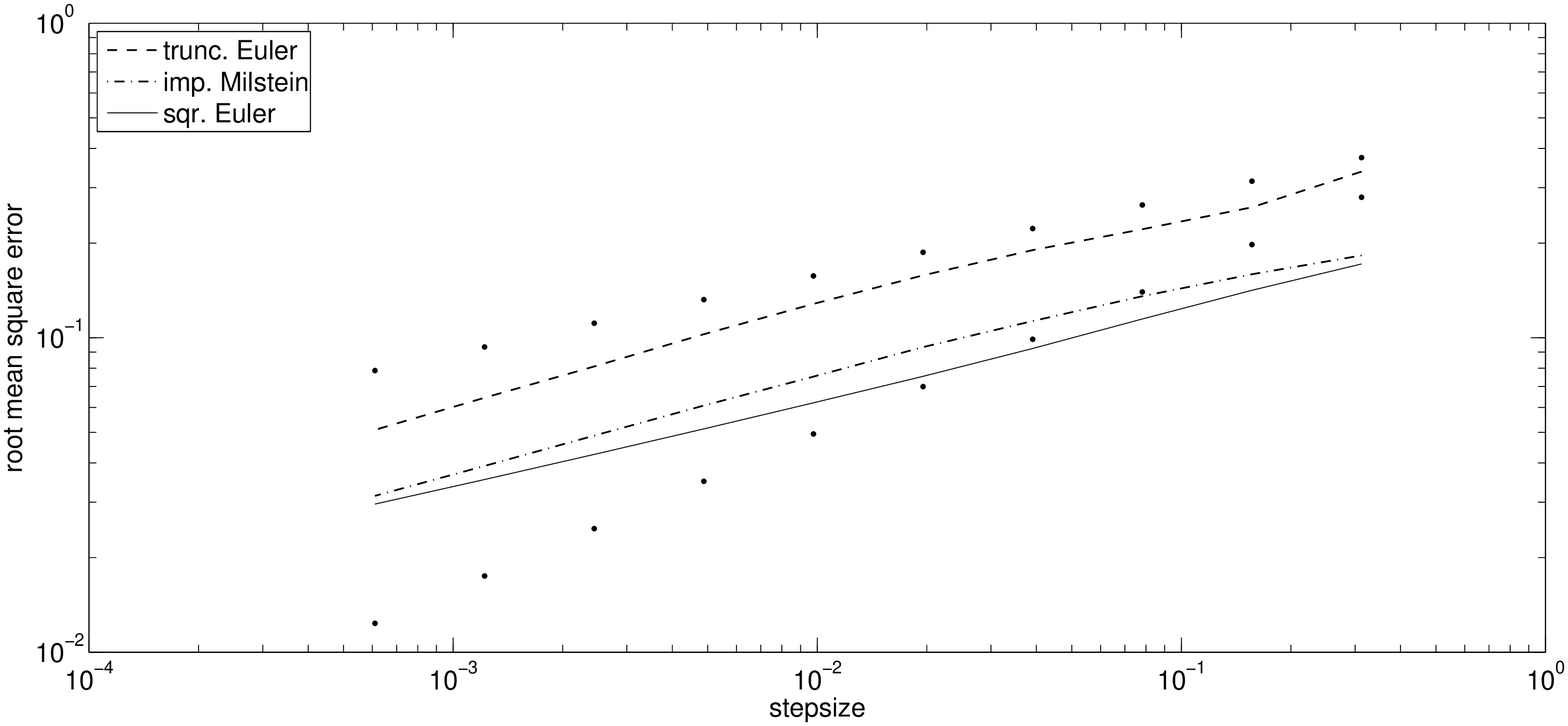, width=13cm,
height=8cm}  
\caption{\label{fig_cir_path_msq_2}Root mean square errors  vs.~step size for Scenario II of the CIR process}
\end{figure}

 For Scenario II, these convergence orders deteriorate and   for all schemes are significantly lower than one half, see also the following table, where   the convergence orders 
have been estimated by a linear regression.

\smallskip

\begin{center}
\begin{tabular}{c||c|c|c|c|} 
empirical conv. order / for & Sc. I 'part' & Sc. II 'part'  & Sc. I 'full' & Sc. II 'full' \\ \hline \hline
truncated Euler   & 0.5739 & 0.3193  &  0.6446 &    0.2960\\ \hline
drift-imp. square-root Euler  & 0.9281  &  0.2734 & 0.8491 &     0.2837 \\ \hline
drift-imp. Milstein &   0.9447 &   0.3096  & 0.8719 &   0.2871
\end{tabular}
\end{center} 
 
 \smallskip
 
Here 'part' denotes the results for the linear regression using only the step sizes  $\Delta=5\cdot 2^{-j}$, $j=7, \ldots, 13$, while 'full' uses the full data set, i.e. the step sizes  $\Delta=5\cdot 2^{-j}$, $j=4, \ldots, 13$. 
 }
  \begin{flushright} $\diamond$ \end{flushright}
\end{example} 
\smallskip

\smallskip
Applying  the Lamperti-transformation also to the asset price in the Heston model gives the log-Heston model
\begin{align*}
d \log(S_t) &= \left( \mu -\frac{1}{2} Y_t^2 \right)   dt + Y_t \left( \sqrt{1-\rho^2} d W^{(1)}_t + \rho dW^{(2)}_t \right), & \quad S_0=s_0 >0  \\ 
 dY_t &= \left( \frac{4 \kappa \lambda - \theta^2}{8}\frac{1}{Y_t}  - \frac{\kappa}{2} Y_t \right) \, dt + \frac{\theta}{2} dW^{(2)}_t, & \quad Y_0=y_0 >0.
\end{align*}
The approximation of the log-Heston price is then a simple integration problem. Using  the Euler scheme for the log-price equation and the drift-implicit square-root Euler scheme for the volatility process yields
an approximation $\overline{H}_{t_k}$ of $\log(S_{t_k})$ given by
$$ \overline{H}_{t_k}= \log(s_0)+ \sum_{\ell=0}^{k-1}\left( \mu -\frac{1}{2} \overline{Y}_{t_{\ell}}^2 \right)\Delta   +  \sum_{\ell=0}^{k-1} \overline{Y}_{t_{\ell}} \left( \sqrt{1-\rho^2} \Delta_{\ell} W^{(1)} + \rho \Delta_{\ell} W^{(2)} \right). $$  This is extended by piecewise linear interpolation to $[0,T]$. 

\smallskip

\begin{corollary} \label{maincor_sqr_euler}
Let $2\kappa \lambda > \theta^2$, $x_0>0$ and $T>0$. Then, for all $$ 1 \leq p < \frac{2 \kappa \lambda}{\theta^2} $$ there exists a constant $K_p >0$   such that
$$	  \left( \EX \max_{t \in [0,T]} |\log(S_{t}) - \overline{H}_t |^p  \right)^{1/p} \leq K_p \cdot \sqrt{| \log(\Delta)|} \cdot \sqrt{\Delta},$$
for all  $\Delta \in (0, 1/2]$.
\end{corollary}

\smallskip
Note that in the Heston model moment explosions may appear according to the parameters of the SDE. In particular, for $p>1$ one has
$ \mathbf{E}S_t^p < \infty$ for all $t>0$ if and only if
$$ \rho \leq  -\frac{\sqrt{p-1}}{\sqrt{p}} + \frac{\kappa}{\theta p}. $$
For more details see e.g. \cite{and_pit}.  Whether this  phenomenon also arises for discretization schemes for the Heston model is unknown at the time of writing.

\smallskip

\begin{example}  {\rm
In this example we test the efficiency of  the Multi-level Monte Carlo estimator $\widehat{P}_{ml}$ see \cite{g1,g2},
based on the above approximation scheme  for the valuation of a European Call option, i.e. for $$ p = e^{-rT} \mathbf{E} (S_T-K)^+.$$
The parameters for the Heston model are
\begin{align*}
& v_0=0.05, \quad \kappa=5.07, \quad \lambda=0.0457,\quad  \theta=0.48, \quad T=1 \\ &   s_0=100, \quad \mu=r=0.0319, \quad\rho=-0.7,  \quad K=105.\end{align*} (Since the riskfree measure is used for the valuation we have $\mu=r$.) In view of the above convergence result for the log-Heston model, we use
the number of levels $L=\lceil \log_{2}(T \varepsilon^{-1}) \rceil$ and the number  of repetitions $N_l = \lceil  L \varepsilon^{-2} T  2^{-\ell} \rceil$, $\ell=0, \ldots, L,$ for a given input  accuracy $\varepsilon >0$, see \cite{g1}. 

The table below shows the empirical root mean square error 
 $$ \textrm{rmsq}=\sqrt{\frac{1}{M}\sum_{i=1}^{M}  |  p - \widehat{P}^{(i)}_{ml}|^2}$$
for the Multi-level estimator  versus the required number of total Euler steps. The latter is proportional to the overall  computational cost of the estimator, i.e.  the number of used random numbers, number of function evaluations and number of arithmetic operations. The $\widehat{P}^{(i)}_{ml}$ are iid copies of the Multi-level estimator $\widehat{P}_{ml}$ and we use $M=5 \cdot 10^4$.
The reference value  $p=7.46253$ was obtained by a numerical evaluation of its Fourier transform representation, see e.g. \cite{CM}. 

For comparison, we also provide the corresponding numerical data for the standard Monte Carlo estimator $\widehat{P}_{st}$, see \eqref{MC_st}, for which we use the relation $\Delta^2 = T/N$ to match stepsize $\Delta$ and numbers of repetitions $N$.
For the same parameters as above the empirical root mean square error of $\widehat{P}_{st}$  is again estimated using $M=5\cdot 10^4$ repetitions.

\bigskip

\begin{center}
\begin{tabular}{c||r|c||r|c|}
$\varepsilon$ & Euler steps of $\widehat{P}_{ml}$ & $\textrm{rmsq}_{\textrm{emp}} $ of $\widehat{P}_{ml}$ &  Euler steps of $\widehat{P}_{st}$  & $\textrm{rmsq}_{\textrm{emp}} $ of $\widehat{P}_{st}$  \\  \hline \hline
$2^{-3}$ & 1056 &  1.369616  &  512 &   1.444497 \\ \hline
$2^{-4}$ & 7168 &  0.685299  & 4096 &    0.714207  \\ \hline
$2^{-5}$ & 43520 & 0.352762  &  32768  &  0.357962   \\ \hline
$2^{-6}$ & 245760  &  0.181384 &  262144  &    0.179231 \\  \hline
$2^{-7}$ & 1318912 & 0.093485  & 2097152 &  0.089618 \\ \hline
$2^{-8}$ & 6815744 & 0.047139 & 16777216  &  0.044821 \\ \hline
\end{tabular}
\end{center}

\bigskip
The numerical data are in good accordance with the predicted convergence behavior, that is
\begin{itemize}
 \item for the Multi-level estimator a root mean square error of order $\varepsilon$ for a computational cost of order $\varepsilon^{-2} |\log(\varepsilon)|^2$ 
\item and for the standard estimator a root mean square error of order $\varepsilon$ for a computational cost of order $\varepsilon^{-3}$. 
\end{itemize}
In particular halving the input accuracy leads  for both estimators (approximately) to a halving of the empirical root mean square error. Moreover, these results illustrate nicely the superiority
of the Multi-level estimator for small input accuracies.

}\end{example}

\bigskip
\bigskip

\section{Summary and Outlook}

In this article we gave a survey on recent results on the convergence of numerical methods for stochastic differential equations in mathematical finance. The presented results include:

\begin{itemize}
\item the pathwise convergence of general It\^o-Taylor schemes for strictly positive SDEs with smooth but not globally Lipschitz coefficients
(Section 2);
\item the construction of  structure, i.e. positivity, preserving approximation schemes  (Sections 2 and 5);
\item the strong convergence of Euler-type methods  for the CEV model and the CIR process (Section 3);
\item the explosion of the moments of the Euler scheme for SDEs for the 3/2-model (Section 3);
\item the strong convergence of the drift-implicit Euler scheme for the Ait-Sahalia model (Section 4);
\item strong convergence rates for the approximation of the CIR and the log-Heston model using a drift-implicit Euler-type method (Section 5).
\end{itemize}
However many unsettled questions are remaining: the exact strong convergence rate of the Euler scheme for the CEV  and CIR processes,  the existence or non-existence of moment explosions for approximation schemes of the Heston model, how to prevent moment explosions (if they happen) by simple modifications of the scheme etc. And even if these questions are answered, the question remains whether there is a 'general theory' for numerical methods 
for SDEs from mathematical finance or do these SDEs have to analysed one by one.
 So, the numerical analysis of  SDEs arising in finance will be still an active and challenging field of research in the future.

\bigskip
\bigskip

\bigskip
\bigskip

{\bf{Acknowledgements.}} The authors would like to thank Martin Altmayer, Martin Hutzenthaler and Arnulf Jentzen for valuable comments
and remarks on an earlier version of the manuscript. Moreover, the authors would like to thank  Mike Giles for a helpful discussion
concerning the numerical evaluation of  Fourier transforms.

\end{document}